\documentclass[11pt]{amsart}
\usepackage{amsmath}
\usepackage{amssymb}
\usepackage{caption}
\usepackage{subcaption}
\usepackage{float}
\usepackage[curve]{xypic}
\usepackage{cite}
\usepackage{enumitem}
\usepackage{color}
\usepackage{url}
\usepackage[usenames,dvipsnames]{xcolor}
\usepackage{graphicx}
\usepackage{verbatim}
\usepackage{tikz}
\usepackage[normalem]{ulem}
\usepackage{hyperref}

\theoremstyle{plain}
\newtheorem{Theorem}{Theorem}[section]

\newtheorem{Definition}[Theorem]{Definition}

\newtheorem{Lemma}[Theorem]{Lemma}
\newtheorem{Proposition}[Theorem]{Proposition}

\newtheorem{conj}[Theorem]{Conjecture}

\newtheorem{Remark}{Remark}

\usetikzlibrary{patterns,arrows.meta}
\usetikzlibrary{arrows.meta}
\usepackage{pgfplots}

\newcommand{\Tei}{\mathcal{T}}
\newcommand{\HP}{\mathbb{H}}
\newcommand{\C}{\mathbb{C}}
\newcommand{\Mod}{\operatorname{Mod}}

\usepackage{microtype}

\title[Carath\'{e}odory metric]{The Carath\'{e}odory metric on Teichm\"uller space of genus two surface}
\author[K. Lin and W. Su]{Kejie Lin and Weixu Su}
\address{Kejie Lin: School of Mathematics, Sun Yat-sen University, Guangzhou 510275, China} \email{linkj23@mail2.sysu.edu.cn}

\address{Weixu Su: School of Mathematics, Sun Yat-sen University, Guangzhou 510275, China}
\email{suwx9@mail.sysu.edu.cn}

\begin{document}
\maketitle

\begin{abstract}
   Let $\Tei_{g,n}$ be the Teichm\"uller space of Riemann surfaces of genus $g$
   with $n$ marked points.  It is conjectured  that the Teichm\"uller and Carath\'{e}odory  metrics coincide on a Teichm\"{u}ller disk exactly when every zero of the associated holomorphic quadratic differential (away from marked points) has even order. Gekhtman and Markovic verified this conjecture for $\Tei_{0,5}\cong \Tei_{1,2}$. We confirm the conjecture  for
   $\Tei_{0,6}\cong\Tei_{2,0}$.
\end{abstract}

\section{Introduction}

Let $\Tei_{g,n}$ be the Teichm\"{u}ller space of Riemann surfaces of genus $g$ with $n$ marked points. 
Throughout this paper, we assume that $3g-3+n\geq 2$. 
The Teichm\"uller space $\Tei_{g,n}$ is the (orbifold) universal cover of the moduli space of Riemann surfaces
 $\mathcal{M}_{g,n}$ and  is naturally a complex manifold of dimension $3g-3+n$. It is known that
 $\Tei_{g,n}$ can be realized as a bounded domain in $\C^{3g-3+n}$, by the Bers
 embedding \cite{Bers}. 
 
Kobayashi and Carath\'eodory metrics are intrinsic metrics on  complex manifolds for which the  Schwarz-Pick inequality holds. This paper is concerned with the relationship between these two metrics on Teichm\"{u}ller spaces.
Let $\mathbb{H}$ be the upper half-plane, equipped with the Poincar\'e metric $$d_{\mathbb{H}}= \frac{|dz|}{2 \operatorname{Im} z}.$$
The \emph{Kobayashi metric} $d_{\mathcal{K}}$ on $\Tei_{g,n}$ is the largest metric so that every holomorphic map $$f: \mathbb{H}\rightarrow\Tei_{g,n}$$ is nonexpanding.
Royden \cite{Royden1969} proved that the Kobayashi metric on $\Tei_{g,n}$ coincides with the Teichm\"{u}ller metric $d_{\mathcal{T}}$. 
Moreover, Royden \cite{Royden1969} showed that the Teichm\"uller metric is not homogeneous at any point and its isometry group is essentially the mapping class group.

On the other hand, the \emph{Carath\'{e}odory metric} is defined as the smallest metric so that every holomorphic map $$F:  \Tei_ {g,n} \rightarrow \mathbb{H}$$ is nonexpanding.
Let $d_\mathcal{C}$ denote the \emph{Carath\'{e}odory metric} on $\Tei_{g,n}$.
 By the Schwarz Lemma, the inequality 
$$d_\mathcal{C} \leq d_\Tei$$ holds  on $\Tei_{g,n}$. 

Teichm\"uller disks, also known as complex geodesics, are holomorphic isometric
 embeddings of the hyperbolic plane into Teichm\"uller space
 (equipped with the Teichm\"uller metric). There is a unique Teichm\"uller disc through any pair of distinct points in $\Tei_{g,n}$. Every Teichm\"uller disk is generated by a non-zero holomorphic quadratic differential $\phi$, with at most simple poles at the marked points.  

Let $$\tau^\phi: \mathbb{H} \rightarrow \Tei_{g,n}$$ 
be a Teichm\"uller disk. 
Kra \cite{Kra1981} proved that the two metrics $d_\Tei$ and $d_\mathcal{C}$ agree on $\tau^\phi$ if all the zeros of $\phi$ are of even order, 
see also McMullen \cite[Theorem 4.1]{McMullen2003}. 
 Note that, in Kra's proof,  if a holomorphic quadratic differential $\phi$ vanishes at a marked point, then it is not necessary to require that it vanish of even order.

A longstanding open  problem is whether the two metrics $d_\Tei$ and $d_\mathcal{C}$ agree on $\Tei_{g,n}$. Recently, Markovic \cite{Markovic2018} solved the problem by showing that  $d_\Tei\neq d_\mathcal{C}$ on $\Tei_{g,n}$. In fact, 
Markovic \cite{Markovic2018} established a criterion that says when the two metrics $d_\Tei$ and $d_\mathcal{C}$ agree on a Teichm\"uller disk. 
Then he used this criterion to show that  $d_\Tei$ and $d_\mathcal{C}$ disagree on certain Teichm\"uller disk in $\Tei_{0,5}$.
Since there exists a holomorphic and isometric embedding of $\Tei_{0,5}$  into $\Tei_{g,n}$, 
we have  $d_\Tei\neq d_\mathcal{C}$ on $\Tei_{g,n}$.
  An application of Markovic's result is to prove a folklore conjecture of Siu that the Teichm\"uller space is not biholomorphic to any bounded convex domain in $\C^{3g-3+n}$. 

\medskip
 
The converse Kra problem is to classify Teichm\"{u}ller disks for which $d_\Tei = d_\mathcal{C}$.
This is equivalent  to classifying Teichm\"{u}ller disks that are holomorphic retracts of the Teichm\"uller space. See Section \ref{sec:disk}. In this paper,
we adopt the following definition.
\begin{Definition}[Abelian Teichm\"uller disk]
   A holomorphic quadratic differential is called \emph{Abelian} if all the zeros
not located at the marked points are of even order.
\end{Definition}
Gekhtman and Markovic   \cite{GM2020} made the following conjecture.

\begin{conj}\label{conj:equl}
A Teichm\"uller disk  is a holomorphic retract if and only if it is generated by
an Ablian holomorphic quadratic differential.
\end{conj}

Gekhtman and Markovic \cite[Theorem 1.2]{GM2020} proved the above conjecture for $\Tei_{0,5}\cong\Tei_{1,2}$. 
In this paper, we confirm the conjecture for $\Tei_{0,6}\cong\Tei_{2,0}$,
which is of complex dimension $3$. 

\begin{Theorem}\label{thm:main2}
Let $\tau^{\phi}$ be a Teichm\"{u}ller disk in $\Tei_{0,6}$
associated to a holomorphic quadratic differential $\phi$.
    The Teichm\"uller metric and Carath\'{e}odory metric agree on $\tau^{\phi}$  if and only if  $\phi$ is Abelian.   
 \end{Theorem}

 Equivalently, we have
\begin{Theorem}\label{thm:main}
Let $\tau^{\phi}$ be a Teichm\"{u}ller disk in $ \Tei_{2,0}$.
    The Teichm\"uller metric and Carath\'{e}odory metric agree on $\tau^{\phi}$  if and only if  $\phi$ is a quadratic differential all of whose zeros are of even order.  
 \end{Theorem}

The main step to prove Theorem \ref{thm:main2} is to establish the 
the following result,  which holds for 
general holomorphic quadratic differential $\phi$ with simple zeros
on $S_{0,n}$, for any $n\geq 5$.

\begin{Theorem}\label{thm:main3}
Assume that $n\geq 5$. 
Let $\phi$ be a holomorphic quadratic differential 
on $S_{0,n}$ with $n-4$ simple zeros and $n$ simple poles (at the marked points). 
Then the associated Teichm\"uller disk $$\tau^\phi: \mathbb{H} \rightarrow \Tei_{0,n}$$ is not a holomorphic retract of $\Tei_{0,n}$.
Equivalently, 
$d_\mathcal{C}< d_\Tei$ on $\tau^{\phi}$.
 \end{Theorem}

 We briefly sketch the proof of Theorem \ref{thm:main3}.
Let $\mathcal{Q}=\mathcal{Q}(1^{n-4},-1^{n})$ be the stratum of quadratic differentials on $S_{0,n}$  with $n-4$ simple zeros and $n$ simple poles.
It is known that $\mathcal{Q}$ is connected \cite{Lanneau}. 
Let $\mathcal{M}_{\mathcal{Q}}$ be the locus of quadratic differentials $\phi\in \mathcal{Q}$
so that the Teichm\"uller disk $\tau^\phi$ is a holomorphic retract of $\mathcal{T}_{0,n}$. As shown by Apisa--Wright \cite[Corollary 1.3]{AW2022}, $\mathcal{M}_{\mathcal{Q}}$ is either empty, the whole stratum $\mathcal{Q}$,  or a full locus of covers of a
stratum of genus zero quadratic differentials.


In Section \ref{free orbit closure}, we will show that the proof of Theorem \ref{thm:main3} can be reduced, via the Apisa–Wright theorem, to the case of staircase surfaces.
In Section \ref{sec:area}, we prove 

\begin{Theorem}\label{thm:staircase}
Assume that $n\geq 5$. 
Let $\phi$ be a staircase surface in $\mathcal{Q}(1^{n-4},-1^{n})$.
Then the Teichm\"uller disk $\tau^{\phi}$ associated to $\phi$
is not a holomorphic retract of $\mathcal{T}_{0,n}$. 
 \end{Theorem}
The proof of Theorem \ref{thm:staircase} is by contradiction.
Assume that $\phi$ is a staircase surface and that the corresponding Teichm\"uller disk $\tau^{\phi}$
has a holomoprhic retraction, say a holomorpic map
$F: \mathcal{T}_{0,n} \to \mathbb{H}$ satisfying $ F\circ \tau^\phi =\operatorname{id}_{\mathbb{H}}.$
In Section \ref{sec:area}, we use the Schwarz-Christoffel formula to compute $F$ on a certain smooth locus of $\mathcal{T}_{0,6}$. The computation is explicit and it shows that  $F$ fails to be $C^2$-smooth, which leads to a contradiction. 
Our computational method generalizes Markovic's previous work on $L$-shaped pillowcases \cite{Markovic2018}.

For the case of $S_{0,6}$, further analysis allows us to classify all Teichm\"uller disks and obtain Theorem \ref{thm:main2}.

\begin{Remark}
In the first draft of the paper, we gave a direct proof of the Theorem \ref{thm:main3} in the case of $S_{0,6}$. We show that 
every Jenkins-Strebel differential in $\mathcal{Q}(1^{2},-1^{6})$
can be deformed (via cylinder deformation or $\operatorname{GL}_2^+(\mathbb{R})$ action) into a staircase surface. 
Although the proof itself does not need the Apisa–Wright theorem, it is rather lengthy. 
\end{Remark}

\begin{Remark}
The calculation to prove Theorem 
 \ref{thm:staircase} can be simplified by degenerating the staircase surface to be an $L$-shaped pillowcase with extra marked points.
We adopt the present calculation in the hope that it is useful in its own right.

\end{Remark}

\subsection*{Acknowledgements} The authors are grateful to Lixin Liu for for his valuable suggestions. W. Su is partially
 supported by NSFC Grant No. 12371076.

\section{Teichm\"{u}ller disks and Orbit closures}

In this section, we introduce the notion of Teichm\"uller disk.  For completeness of exposition,
we shall recall the method of  Gekhtman and Markovic \cite{GM2020},
which reduces the classification of Teichm\"uller disks generated by general quadratic differentials to  Jenkins-Strebel differentials. 
For more background on Teichm\"uller theory and quadratic differentials, we refer to 
Hubbard \cite{Hubbard} and Strebel \cite{Strebel}.

\subsection{Teichm\"{u}ller disks and holomorphic retracts}
\label{sec:disk}
Let $S_{g,n}$ be an oriented surface of genus $g$
with $n$ marked points, where $3g-3+n\geq 2$. Let $\Tei_{g,n}$ be the Teichm\"{u}ller space of Riemann surfaces marked by $S_{g,n}$.
Denote by $\Mod_{g,n}$ the mapping class group of $S_{g,n}$. The quotient $$\mathcal{M}_{g,n}=\Tei_{g,n}/\Mod_{g,n}$$ is the moduli space of Riemann surfaces. Let
 $$\pi:\Tei_{g,n}\rightarrow \mathcal{M}_{g,n}$$ be the natural projection.

Let $X\in \Tei_{g,n}$. A \emph{holomorphic quadratic differential} $\phi$ on $X$ is a $(2,0)$ tensor
locally given by 
$\phi=\phi(z) dz^2,$
where $\phi(z)$ is a holomorphic function. We allow $\phi$ to have at worst simple (first order) poles at the marked points. 
This is equivalent to require that the  $L^1$-norm
$$\|\phi\| = \int_X |\phi|$$
is finite. When a marked point is a zero or pole, we call it singular. Otherwise, it is called  regular.  Any holomorphic quadratic differential 
induces a flat metric of finite area, which is also called a \emph{half-translation surface}. 

Let $Q(X)$ be the space of holomorphic quadratic differentials  on $X$.
It is known that the  cotangent space of $\Tei_{g,n}$ at $X$
can be naturally identified with $Q(X)$ and the $L^1$-norm is dual to the Teichm\"uller norm. 
We denote the cotangent bundle by $Q\Tei_{g,n}$. 
A pair $(X,\phi)\in Q\Tei_{g,n}$ with $\phi\in Q(X)$ and $\phi\neq 0$ generates a holomorphic embedding
$$\tau^\phi: \HP \rightarrow \Tei_{g,n},$$ 
which is an isometry from the Poincar\'e metric on $\HP$
to the Teichm\"uller metric on $\Tei_{g,n}$. For each $\lambda\in \mathbb{H}$, 
the Riemann surface $X_\lambda= \tau^\phi(\lambda)$ is characterized by the property that the extremal quasiconformal map between $X$ and $X_\lambda$ has Beltrami differential 
$$\mu_\lambda = \left( \frac{i-\lambda}{i+\lambda}  \right) \frac{|\phi|}{\phi}.$$
In the natural coordinates of $\phi$, the extremal quasiconformal map is 
an affine map of the form $$z=x+iy \mapsto  x+\lambda y.$$ 
By abuse of notation, we write $$\tau^{\phi}(\lambda)=\left( \frac{i-\lambda}{i+\lambda}  \right) \frac{|\phi|}{\phi}.$$ 
We call $\tau^{\phi}$ the \emph{Teichm\"{u}ller disk} generated by $\phi$.
Note that $\tau^\phi\left( i \cdot \mathbb{R}_+ \right)$ is the Teichm\"uller geodesic
in the direction of $\phi$.

\begin{Definition}[Holomorphic retract] \label{def:retract}
  Let $\tau^\phi: \HP \rightarrow \Tei_{g,n}$ be a Teichm\"uller disk. 
We say that $\tau^{\phi}$ is a \emph{holomorphic retract} of $\Tei_{g,n}$ if 
there exists a holomorphic map $F : \Tei_{g,n} \to \HP$ such that $$F\circ \tau^\phi=\mathrm{id}_\HP.$$
Any holomorphic map  $F : \Tei_{g,n} \to \HP$ satisfies the above property is called a holomorphic retraction of $\tau^\phi$.
\end{Definition}

The following lemma states that classifying Teichm\"uller disks for which $d_\Tei = d_\mathcal{C}$ is equivalent to classifying holomorphic retracts of $\Tei_{g,n}$.
See \cite[Lemma 1.3]{GM2020}.
\begin{Lemma}
 The Kobayashi and Carath\'{e}odory metrics agree on a Teichm\"{u}ller disk $\tau^{\phi}$ if and only if $\tau^{\phi}$ is a holomorphic retract of $\Tei_{g,n}$.
\end{Lemma}

\subsection{Orbit closures}

There is a natural $\operatorname{GL}_2^+(\mathbb{R})$-action on $Q\Tei_{g,n}$ 
that commutes with 
the action of $\operatorname{Mod}_{g,n}$.  Intuitively, one can define a quadratic differential
$\phi$ by gluing a polygon $P$ along sides in pairs (the sides in each pair are parallel and of the
same length). If $A\in \operatorname{GL}_2^+(\mathbb{R})$, then $A$ acts on $P$
as an affine transformation. By definition,  $A \cdot \phi$ is the quadratic differential obtained by gluing $A\cdot P$ in the same pattern as that for $P$. 

Let $p: Q\Tei_{g,n} \to \Tei_{g,n}$ be the natural projection. 
For any $(X,\phi) \in  Q\Tei_{g,n}$, the image of $A \cdot (X,\phi)$ under $p$
remains the same if and only if we multiply $A$ by $\mathbb{C}^*=\mathbb{C}\setminus\{0\}$. 
Note that $\mathbb{H} \cong  \mathbb{C}^* \backslash \operatorname{GL}_2^+(\mathbb{R})$. In fact, we can identify $\mathbb{H}$
as a subgroup of $\operatorname{GL}_2^+(\mathbb{R})$ :

$$\mathbb{H} \cong\left\{\begin{pmatrix}
    1 & \operatorname{Re} \lambda\\
    0 & \operatorname{Im} \lambda
\end{pmatrix} \ : \  \lambda\in \mathbb{H} \right\}.$$

In complex coordinates, the action of $\begin{pmatrix}
    1 & \operatorname{Re} \lambda\\
    0 & \operatorname{Im} \lambda
\end{pmatrix}$ is given by 
$$x+i y \mapsto x + \lambda y.$$
So the Beltrami differential is $$\left( \frac{i-\lambda}{i+\lambda}  \right) \frac{|\phi|}{\phi},$$
which coincides with our definition of $\tau^\phi(\lambda)$  before.

Let $Q\mathcal{M}_{g,n} = Q\Tei_{g,n} / \Mod_{g,n}$ be the moduli space of holomorphic quadratic differentials. 
We also use $\pi$ to denote the natural projection $ Q\Tei_{g,n} \to Q\mathcal{M}_{g,n}$. The $\operatorname{GL}_2^+(\mathbb{R})$-action on $Q\Tei_{g,n}$ 
induces an action on $Q\mathcal{M}_{g,n}$.
Given $(X,\phi)\in Q\mathcal{M}_{g,n}$, we denote by $\operatorname{GL}_2^+(\mathbb{R}) \cdot \phi$ the $\operatorname{GL}_2^+(\mathbb{R})$-orbit of $(X,\phi)$. 
Let $\overline{\operatorname{GL}_2^+(\mathbb{R}) \cdot \phi}$ be the closure of the orbit in $Q\mathcal{M}_{g,n}$.

The next result is given by Gekhtman and Markovic \cite[Proposition 2.1]{GM2020}.

\begin{Proposition}\label{lem:GM-orbit}
Let $\phi$ be a quadratic differential. 
    If $\phi$ generates a Teichm\"uller disk $\tau^\phi$ that is a holomorphic retract of $\mathcal{T}_{g,n}$, so does every  element in  $\overline{\operatorname{GL}_2^+(\mathbb{R}) \cdot \phi}$.
\end{Proposition}
In fact, Gekhtman and Markovic \cite[Lemma 2.2]{GM2020} proved
\begin{Proposition}\label{limits}
Let $\phi_k$ be a sequence of quadratic differentials in $Q\mathcal{T}_{g,n}$ that converges to a nonzero quadratic differential $\phi$. 
    If each $\tau^{\phi_k}$ is a holomorphic retract, then so is  $\tau^{\phi}$.
\end{Proposition}

Let 
$ \kappa =(\kappa_1, \cdots, \kappa_m)$ be an integral vector where $\kappa_i \geq -1$, $i = 1,...,m$, and satisfy $\sum\limits_{i=1}^{m} \kappa_i =4g-4.$ 
Denote by $\mathcal{Q}(\kappa)$ the stratum of quadratic differentials $(X,\phi)$ in
$Q\mathcal{M}_{g,n}$ consisting of $\phi$ with $m$ distinct zeros or poles of multiplicities $\kappa_1, \cdots, \kappa_m$. 
Here we allow $\kappa_i=0$, which is corresponding to a regular marked point. In addition, $\phi$ may have at
worst simple poles at the marked points, which correspond to those $\kappa_i$ equal to $-1$. 

The horocycle flow action $h_t : \mathcal{Q}(\kappa) \rightarrow \mathcal{Q}(\kappa)$ 
is defined  as the restriction of the $\operatorname{GL}_2^+(\mathbb{R})$-action to the subgroup
$$H=\left\{\begin{pmatrix}
    1 & t\\
    0 & 1
\end{pmatrix} \ : \  t\in \mathbb{R} \right\}.$$
We refer to \cite{CW2022} for a comprehensive survey of the dynamics of
 the horocycle flow.

Let $\phi$ be a non-zero quadratic differential  in $\mathcal{Q}(\kappa)$.  
Let $\mathcal{F}(\phi)$ be the horizontal foliation
 of $\phi$, whose leaves are tangent to the directions  such that $\phi(z)dz^2>0$.   Denote the horizontal critical graph of $\phi$
by $\Gamma=\Gamma(\phi)$. 
We say that  $\phi$ or $\mathcal{F}(\phi)$ is \emph{Jenkins-Strebel}  if $\Gamma(\phi)$ is compact and the complement $X\setminus \Gamma(\phi)$  is a union of disjoint flat cylinders foliated by horizontal closed leaves.

Smillie and Weiss \cite[Theorem 5]{SW2004} proved the following theorem. See also \cite[\S 2.4]{GM2020}.

\begin{Theorem}\label{thm:SW}
Let $(X,\phi) \in \mathcal{Q}(\kappa)$. Then
\begin{enumerate}
  \item The  closure $\overline{H \cdot \phi}$ in $\mathcal{Q}(\kappa)$ contains a Jenkins-Strebel differential $\psi\in  \mathcal{Q}(\kappa)$. 
  \item If $\phi$ has $k$ horizontal cylinders whose union is not dense in $S_{g,n}$,
  then there exists a Jenkins-Strebel differential $\psi$ in $\overline{H \cdot \phi}$ that has at least $k+1$ horizontal cylinders. 
\end{enumerate}   
\end{Theorem}

By Proposition \ref{lem:GM-orbit} and Theorem \ref{thm:SW}, Conjecture \ref{conj:equl}
can be reduced to the classification of Jenkins-Strebel differentials.

\subsection{Markovic's criterion } 
Suppose that $\phi$ is a Jenkins-Strebel differential on $X$ that decomposes the surface into a finite number of horizontal cylinders $\Pi_{j}, j=1,\cdots,k$. 

Let $\mathbb{H}^k$ be the $k$-fold product of the upper half-plane $\mathbb{H}$. Following \cite[Section 4]{Markovic2018}, we define the \emph{Teichm\"uller polydisk} $$\mathcal{E}^\phi:\mathbb{H}^k\rightarrow \Tei_{g,n}$$ by associating 
each ${\lambda}= (\lambda_1, \cdots, \lambda_k)\in \mathbb{H}^k$ a Riemann surface $X_\lambda\in \Tei_{g,n}$ determined by the Beltrami differential  $$\mu_\lambda=\left( \frac{i-\lambda_j}{i+\lambda_j} \right) \frac{|\phi|}{\phi}$$ 
on each $\Pi_{j}$. If we take $\lambda\in \mathbb{H}^k$ such that all $\lambda_i$ are equal, then we obtain the Teichm\"uller disk generated by $(X,\phi)$.

\begin{Remark}
  It is known that $\mathcal{E}^\phi$ is a holomorphic embedding, but not proper.
   See \cite[Section 4]{Markovic2018}. 
\end{Remark}

Denote by $h_j$ the height of $\Pi_{j}$.
Note that each $\Pi_j$ is deformed into a new horizontal cylinder $\Pi_j(\lambda)$.
The circumference is preserved, while the height is scaled by a factor $\operatorname{Im} \lambda_j$. The real part 
$\operatorname{Re} \lambda_j$ measures the shear of the cylinder. 
By gluing $\Pi_j(\lambda), j =1, \cdots, k$, we obtain a
new Jenkins-Strebel differential $\phi_\lambda$ on $X_\lambda$.

The following beautiful and important result provides a criterion to characterize Jenkins-Strebel differentials which generate Teichm\"uller disks 
that are holomorphic retracts.  See \cite[Theorem 4.1]{GM2020}.

\begin{Theorem}[Markovic's criterion]\label{thm:GM-JS}
      Let $\phi\in Q(X)$ be a  Jenkins-Strebel differential. Denote the area of the $j$-th cylinder of $\phi$ by $\alpha_j$ . Let $\mathcal{E}^{\phi}:\mathbb{H}^k\rightarrow\Tei_{g,n}$ be the Teichm\"{u}ller polydisk associated with $\phi$. Then the Teichm\"{u}ller disk $\tau^{\phi}$ is a holomorphic retract  if and only if there exists a holomorphic map $\Phi:\Tei_{g,n}\rightarrow \mathbb{H}$ satisfying \begin{equation}\label{equ:criterion}
       \left(\Phi\circ \mathcal{E}^\phi\right)(\lambda)=\sum\limits_{j=1}^{k}\alpha_j\lambda_j.
      \end{equation}
\end{Theorem}

As a corollary of Theorem \ref{thm:GM-JS}, if $\phi$ is a Jenkins-Strebel differential with $k$ cylinders and if $\phi$ generates a Teichm\"uller disk that is a holomorphic retract, then so is every quadratic differential in its $\mathbb{H}^k$-orbit.



    



\section{Classifying quadratic differentials}\label{free orbit closure}

\subsection{Quadratic differential in the principal stratum}

Let $\mathcal{Q}$ be a stratum of quadratic differentials. 
Let $\mathcal{M}_\mathcal{Q}$ be the locus of  $(X,q)\in\mathcal{Q}$ so that
the Teichm\"uller disk generated by $(X,q)$ is a holomorphic retract of
the Teichm\"uller space. 
Let $\mathcal{M}$ be a connected component of $\mathcal{M}_\mathcal{Q}$.

By \cite{GM2020,AW2022}, $\mathcal{M}$ is \emph{free}. 
The means that, for any $(X,q)\in \mathcal{M}$,  any cylinder deformation of $(X,\omega)$ 
 remains in $\mathcal{M}$.
The following result of Apisa--Wright \cite{AW2022} classifies $\mathcal{M}$.

\begin{Theorem}\label{thm:Apisa-Wright}
    $\mathcal{M}$ is either empty,  
    $\mathcal{Q}$,  or a  full locus of covers of a stratum of genus zero quadratic differentials.
\end{Theorem}

Let $\mathcal{M}$ and $\mathcal{N}$ be invariant subvarieties of quadratic differentials. 

We say that $\mathcal{M}$ is a full locus of covers of $\mathcal{N}$ if every surface in $\mathcal{M}$ is given by a  cover of a surface in $\mathcal{N}$, and all deformations of surfaces remaining in $\mathcal{N}$ give rise to covering surfaces in $\mathcal{M}$.  More precisely, there is a branched cover $f$ that pulls back
every $(Y,q)\in\mathcal{N}$ to some $(X,\phi)\in \mathcal{M}$.  
We denote $(X,\phi)=f^*(Y,q).$ 
Then $$\mathcal{M}=\{f^*(Y,q) \ : \ (Y,q)\in \mathcal{N}\}.$$ It should be noted that, in such a covering construction, all marked points  on $(X,\phi)$ map to
 marked points on $(Y,q)$, and each regular marked point on $(Y, q)$ has at least one pre-image on
$(X, \phi)$ that is a zero point or marked point, but poles on $(Y,q)$ need not have pre-images that are zeros or marked points on $(X,\phi)$.
See \cite[Definition 3.2]{AW2022} for details. 

Now we take $\mathcal{Q}=\mathcal{Q}(1^{n-4},-1^{n})$, the stratum of 
quadratic differentials on $S_{0,n}$ with $n-4$ simple zeros and $n$ simple poles (at the marked points). For simplicity of notation, we denote $\mathcal{Q}=\mathcal{Q}(1^{n-4},-1^{n})$.
We assume that $n>4$. As before, $\mathcal{M}_\mathcal{Q}$ denotes the locus of quadratic differentials $(X,q)\in\mathcal{Q}$ so that
the Teichm\"uller disk generated by $(X,q)$ is a holomorphic retract of
$\mathcal{T}_{0,n}$.  Let $\mathcal{M}$ be a connected component of $\mathcal{M}_\mathcal{Q}$.

We have the following result, which is equivalent to Theorem \ref{thm:main3}. The proof relies on Theorem \ref{thm:staircase}. We shall assume Theorem \ref{thm:staircase} holds and postpone its proof to Section \ref{sec:area}.

\begin{Proposition}\label{MQ is empty}
      $\mathcal{M}$ is  empty.
\end{Proposition}
\begin{proof}
The proof is by induction on $n$.
The case that $n=5$ is established by Gekhtman-Markovic \cite{GM2020}.
In fact, they show that any $(X,\phi)$ in $\mathcal{Q}(1,-1^5)$ can be deformed into an $L$-shaped pillowcase, and the Teichm\"uller disk generated by any $L$-shaped pillowcase is not a holomorphic retract. 

In general,  Theorem \ref{thm:staircase} confirms that there is a staircase surface in  $\mathcal{Q}(1^{n-4},-1^n)$ such that the associated Teichm\"uller disk  is not a holomorphic retract. As a result, $\mathcal{M}$ could not be the whole stratum $\mathcal{Q}$. 

With the aid of Theorem \ref{thm:Apisa-Wright}, it remains to show that 
 $\mathcal{M}$ is not a  full locus of covers of a stratum of genus zero quadratic differentials. Suppose not, there is a stratum of quadratic differentials $\mathcal{N}$ on a punctured sphere such that 
 $\mathcal{M}$ is a  full locus of covers of $\mathcal{N}$.

 For the subvariety $\mathcal{M}$, since 
$\mathcal{M}\subset \mathcal{Q}(1^{n-4},-1^n)$, $\mathcal{N}$ must 
be contained in $\mathcal{Q}(1^{m-4},-1^m)$.
That is because any $(Y,q)\in \mathcal{N}$ could not have higher order zeros or regular marked points. For otherwise, they would be lifted to  higher order zeros or regular marked points on $(X,\omega)$.
Thus $\mathcal{N}=\mathcal{Q}(1^{m-4},-1^m)$. 

By our construction, for every $(Y,q)\in \mathcal{N}$,
the associated Teichm\"uller disk $$\tau^q: \mathbb{H}\to \mathcal{T}_{0,m}$$
lifts to the Teichm\"uller disk $$\tau^\phi: \mathbb{H}\to \mathcal{T}_{0,n},$$
where $(X,\phi)=f^*(Y,q).$ Since $\tau^\phi$ is a holomorphic retract,
there exists a holomorphic mapping $F: \mathcal{T}_{0,n} \to \mathbb{H}$
so that $F\circ \tau^\phi = \operatorname{id}_{\mathbb{H}}$. 
By identifying $\tau^\phi$ with $\tau^q$, we obtain $F\circ \tau^q = \operatorname{id}_{\mathbb{H}}$. This indicates that, by lifting $\mathcal{T}_{0,m}$ to a totally geodesic submanifold of 
$\mathcal{T}_{0,n}$,  the restriction of 
$F$ on $\mathcal{T}_{0,m}$ defines a holomorphic retraction of $\mathcal{T}_{0,m} $. We conclude that $\mathcal{N}$ is free. 

According to our inductive argument, the only possibility that $\mathcal{N}$ is nonempty and free is that $\mathcal{N}=\mathcal{Q}(-1^4)$,  the stratum of pillowcase surfaces. 
In this case, every $(X,\phi)$ in $\mathcal{M}$ is a pillowcase-tiled surface.
Without loss of generality, we may assume that $(X,\phi)$ is a Jenkins-Strebel differential. It follows that $(X,\phi)$ consists of a single 
horizontal cylinder. For otherwise, we can deform different cylinders separately, which contradicts with the fact that $\mathcal{M}$ is the full locus of covers of $\mathcal{Q}(-1^4)$. 

\begin{figure}[h]
    \centering
    \begin{tikzpicture}[line width=1pt]
        \draw (0,0)--(0,-2)--(-1.5,-2.8);
        \draw (0,-2)--(1.5,-2.8);
        \draw (3.3,0)--(3.3,-2)--(2,-2.8);
        \draw (3.3,-2)--(4.8,-2.8);
        \node[ circle, fill=black,inner sep = 2pt] (A) at (0,-2){};
        \node[ circle, fill=black,inner sep = 2pt] (A) at (3.3,-2){};
         \node  at(0,0){$\times$}; 
        \node  at(-1.5,-2.8){$\times$};
        \node  at(1.5,-2.8){$\times$}; 
       \node  at(3.3,0){$\times$}; 
        \node  at(2,-2.8){$\times$}; 
        \node  at(4.8,-2.8){$\times$}; 
         
\node at (2.3,-5.3) {(a)};
\draw [dashed] (5.7,-4.3)--(5.7,1.2);

                  \draw (6.5,0)--(8,-1.5)--(10.5,-1.5)--(11.5,-0.5);
        \draw (6.2,-3.1)--(8,-1.5);
        \draw (10.5,-1.5)--(12,-3);
        \node[ circle, fill=black,inner sep = 2pt] (A) at (8,-1.5){};
        \node[ circle, fill=black,inner sep = 2pt] (A) at (10.5,-1.5){};
        \draw (7.1,-4.0)--(11,-4.0);
        \node  at(6.5,0){$\times$}; 
        \node  at(11.5,-0.5){$\times$};
\node  at(6.2,-3.1){$\times$}; 
\node  at(7.1,-4.0){$\times$}; 
\node  at(11,-4.0){$\times$}; 
        \node at(12,-3){$\times$};
        \node at (8.8,-5.3) {(b)};
    \end{tikzpicture}
    \caption{Jenkins-Strebel differentials on $S_{0,6}$ with a single cylinder.}
    \label{fig:I-8}
\end{figure}

Now assume that $\phi$ is a Jenkins-Strebel differential with a single horizontal cylinder, then the horizontal critical graph of $\phi$ must consist of
two connected components, each of which is a tree with edges connecting the singularities of $\phi$. 
See Figure \ref{fig:I-8} for an illustration in the case of $S_{0,6}$. 
In this case, we can shear the horizontal cylinder approximately such that there is a vertical saddle connection connecting a pair of poles. 
Denote such a saddle connection by $\gamma$. Then there must be a vertical cylinder with  closed leaves surrounding $\gamma$ and such a vertical cylinder is not dense in the surface. Rotating the quadratic differential by 
$\begin{pmatrix}
    0 & -1\\
    1 & 0
\end{pmatrix}$,
we obtain a new quadratic differential, still denoted by $\phi\in \mathcal{M}$ that contains a horizontal cylinder,which is not dense in the surface. 
By applying Theorem \ref{thm:SW}, there is a Jenkins-Strebel differential 
$\psi$ in $\overline{H\cdot \phi}$ that has a least two horizontal cylinder. 
It is obvious that $\psi\in \mathcal{M}$. As a result, $\psi$ is a pollowcase-tiled surface with at least two horizontal cylinders. This contradicts our previous argument.

\end{proof}

\subsection{Six punctured sphere}
Now we prove Theorem\ref{thm:main2} and Theorem \ref{thm:main}.
Note that  $\mathcal{Q}\mathcal{M}_{(0,6)}$
consists of the following strata:

\begin{enumerate}
    \item[(I)] $Q(1^2,-1^6)$, the stratum of quadratic differentials with two simple zeros and six simple poles;
    \item[(II)] $Q(2,-1^6)$, the stratum of quadratic differentials with a zero of order $2$ and six simple poles;
    \item[(III)] $Q(1,-1^5)$, the stratum of quadratic differentials with a single simple zeros and five simple poles, and the zero locates at a marked point;
    \item[(IV)] $Q(1,0,-1^5)$, the stratum of quadratic differentials with a single simple zeros and five simple poles, and the zero does not locate at a marked point;
    \item[(V)] $Q(0^2,-1^4)$, the stratum of quadratic differentials with no zeros and four simple poles. 
\end{enumerate}

\subsection*{Case (I)} 
This is the special case of Theorem \ref{thm:main3} for $n=6$. In this case, $\tau^\phi$ is not a holomorphic retract of $\mathcal{T}_{0,6}$. 

\subsection*{Case (II), (III) and (V)} 
In these cases, $\tau^\phi$ is a holomorphic retract of $\mathcal{T}_{0,6}$, as proved by  
Kra \cite{Kra1981} and McMullen \cite{McMullen2003}.

\subsection*{Case (IV)} 
In this case, 
the zero of $\phi$ does not lie at a marked point. 
We have to show that $\tau^\phi$ is not a holomorphic retract of $\mathcal{T}_{0,6}$.

We may assume that $\phi$ is a Jenkins-Strebel differential. There are two  possibilities for the critical graph $\Gamma=\Gamma(\phi)$, see Figure \ref{fig:III-1} for an illustration. The right one can be transform into the left one (see \cite[\S 6.1]{GM2020}).

\begin{figure}[h]
\centering

\begin{tikzpicture}[line width=1pt,scale=0.8]
 
  \draw (-4.2,0) ellipse (1.5cm and 2cm);
  \node[ circle, fill=black,inner sep = 2pt] (A) at (-4.2, 2.0){};
  \draw (-4.2,2.0) to (-4.2,3.0);
      \node  at (-4.2, 3){$\times$};
   \draw (-4.7,0) to (-3.7,0);
    \node  at (-4.7, 0){$\times$};
     \node at (-3.7, 0){$\times$};
     \draw (-5.2,-3.3) to (-3.2, -3.3);
      \node at (-5.2, -3.3){$\times$};
     \node  at (-3.2, -3.3){$\times$};
     \draw[dashed] (-2,3)--(-2,-3.5);
     \draw (0,2.5)--(0,0)--(-1.5,-1.3);
        \draw (0,0)--(1.5,-1.3);
     \draw (-1.3,-2.7)--(1.8,-2.8);
\node at (0,2.5){$\times$};
     \node  at (1.5,-1.3){$\times$};
 \node  at (-6, -0.3){$\gamma$};
 \node  at (-6, -2){$\Pi_2$};
  \node  at (-4, -1){$\Pi_1$};
     \node at (-1.5,-1.3){$\times$};
     \node  at (-1.3,-2.7){$\times$};
     \node at (1.8,-2.8){$\times$};
\node[ circle, fill=black,inner sep = 2pt] (A) at (0,0){};
     
\end{tikzpicture}
\caption{Critical graphs Jenkins-Strebel differentials  in $\mathcal{Q}(1,0,-1^5)$.
The extra marked point, considered as a zero of order zero, is not drawn in the picture.}
\label{fig:III-1}
\end{figure}

Let us assume that $\Gamma$ is the one on the left-hand side of Figure \ref{fig:III-1}.
In this situation, we shall consider the regular marked point on $S_{0,6}$ as a singularity of order $0$. 

There is a horizontal closed saddle connection $\gamma$ joining the zero to itself,
which separates the surface into two horizontal cylinders $\Pi_1$ and $\Pi_2$.  
We can use the $\mathbb{H}^2$-action to shear $\phi$ appropriately along the horizontal cylinders so that $\phi$ is an $L$-shaped pillowcase  with an extra marked point. See Figure \ref{fig:III-2} for an illustration, where we denote the zero by $z$, the poles by $p_1, \cdots, p_5$ and the extra marked point by $p_0$. 
Without loss of generality, we may assume that $p_0$ lies in the right-hand side of $\Pi_2$. 
This case can be dealt with using the same proof as Theorem \ref{thm:staircase}  (cf. Section \ref{sec:L-shape}).
The  same computation as in Section \ref{sec:area}  will implies that in this case the associated Teichm\"uller disk 
is not a holomorphic retract of the Teichm\"uller space. It turns out that the 
locus $\mathcal{M}_\mathcal{Q}$ where $Q=\mathcal{Q}(1,0,-1^5)$ is empty.

\begin{figure}[h]
    \centering
   \begin{tikzpicture}[line width=1.5pt,scale=1.5]
        \node[circle, fill=black,inner sep = 2pt] (A) at (0,0){};
       \node  at(0,2){$\times$};
       \node  at(-1,2){$\times$};
      \node at(2,0){$\times$};
      \node at(2,-1){$\times$};
      \node at(-1,-1){$\times$};
      \draw (0,0) -- (2,0);
      \draw (0,0) -- (0,2);
      \draw (0,2) -- (-1,2);
      \draw (-1,2) -- (-1,-1);
      \draw (-1,-1) -- (2,-1);
      \draw (2,-1) -- (2,0);
       \draw (-1,0) to[out=315,in=225] (0,0);
         \draw[dashed] (-1,0) to[out=45,in=135] (0,0);
          \node[circle, fill=blue,inner sep = 2pt] (A) at (2,-0.5){};
          \node[right] at (2,-0.5){$p_0$};
      \node at (0.5,-0.5){$\Pi_2$};
      \node at (-0.5,1){$\Pi_1$};
      \node[right] at (0,0.15) {$z$};
       \node[right]  at (0,2) {$p_1$};
       \node[left]  at(-1,2){$p_2$};
      \node[right] at(2,0){$p_5$};
      \node[below] at(2,-1){$p_4$};
      \node[left] at(-1,-1){$p_3$}; 
   \end{tikzpicture}
   \caption{The $L$-shaped pillowcase with an extra marked point $p_0$ lying at the right-hand side.}
   \label{fig:III-2}
\end{figure}

\bigskip

The preceding discussion proves Theorem \ref{thm:main2}. 
We conclude this section with a proof of Theorem \ref{thm:main}.
There is a canonical double cover $$\eta: S_{2,0} \to S_{0,6}$$
branched over the marked points. By pulling back
Riemann surface structures, $\eta$ induces a holomorphic isometry  of Teichm\"uller spaces (with respect to the Teichm\"uller metric)
$$\eta^* : \Tei_{0,6} \cong \Tei_{2,0}.$$
Moreover, any quadratic differential $\phi$ on $X\in \Tei_{0,6}$ is pulled back  to a 
quadratic differential $\psi$ on $Y= \eta^*(X)$. In the opposite direction, we can identify any quadratic differential 
$\psi$ on a genus two surface with some $\phi$ on $S_{0,6}$. 
Note that if $\psi$ has a zero of odd order, then either $\phi\in \mathcal{Q}(1^2,-1^6)$ or  $\phi\in \mathcal{Q}(1,0,-1^5)$. Thus Theorem \ref{thm:main} follows from Theorem \ref{thm:main2}.

\begin{figure}[h]
    \centering
    \begin{tikzpicture}[line width=1pt,scale=0.75]

        \draw [line width=1pt] (-1.2,4.5) arc (40:320:1.5cm and 1.0cm);       
       \draw [line width=1pt] (-1.2,4.5) arc (230:310:1.5cm and 1.0cm);
        \draw [line width=1pt] (-1.2,3.22) arc (130:50:1.5cm and 1.0cm);     
        \draw [line width=1pt] (+0.7,4.5) arc (140:0:1.5cm and 1.0cm);         
       \draw [line width=1pt] (+0.7,3.22) arc (220:360:1.5cm and 1.0cm);  
      \draw [line width=1pt] (-1.6,4) arc (30:150:0.8cm and 0.5cm);
    \draw [line width=1pt] (-1.6,4) arc (330:190:0.8cm and 0.5cm);
    \draw [line width=1pt] (-1.6,4) arc (330:350:0.8cm and 0.5cm);
     \draw [line width=1pt] (2.5,4) arc (30:150:0.8cm and 0.5cm);
    \draw [line width=1pt] (2.5,4) arc (330:190:0.8cm and 0.5cm);
    \draw [line width=1pt] (2.5,4) arc (330:350:0.8cm and 0.5cm);
    \draw [line width=1pt][-{To[length=3mm,width=2mm]}] 
    (0,3)--(0,1.7);
    \node at (0.2,2.3) {$f$};
    \node [ circle, fill=black,inner sep = 1pt] (A) at (-3.7,3.96) {};
     \node [ circle, fill=black,inner sep = 1pt] (A) at (-3.2,3.96) {};   
    \node [ circle, fill=black,inner sep = 1pt] (A) at (-1.4,3.96) {};
     \node [ circle, fill=black,inner sep = 1pt] (A) at (0.8,3.96) {};
    \node [ circle, fill=black,inner sep = 1pt] (A) at (2.7,3.96) {};
    \node [ circle, fill=black,inner sep = 1pt] (A) at (3.2,3.96) {};
    \draw [line width=0.4pt](-4.3,3.96)--(-3.9,3.96); 
    \draw [line width=0.4pt][->] (-3.94,3.9) arc (350:25:0.1cm and 0.22cm);      
        \draw [line width=1pt](0,0) circle (1.5);
        \draw [line width=0.2pt,dashed] (1.5,0) arc (0:180:1.5cm and 0.4cm);
        \draw [line width=0.2pt] (-1.5,0) arc (180:360:1.5cm and 0.4cm);
        \node at (0.7,0.8){$\times$};
        \node at (0,1){$\times$};
        \node at (-0.7,0.8){$\times$};
        \node at (0.7,-0.8){$\times$};
        \node at (0,-1){$\times$};
        \node at (-0.7,-0.8){$\times$};
    \end{tikzpicture}
        \caption{The canonical double cover $\eta: S_{2,0}\to S_{0,6}$.}
    \label{fig:enter-label}
\end{figure}

    \section{Staircase surfaces}\label{sec:area}

Let us start with some notation. Assume that $\mathbf{a}=\left(a_1, \cdots, a_{n-3}\right)$ and 
$\mathbf{b}=\left(b_1, \cdots, b_{n-3}\right)$, 
where $a_{i}>0, b_{i}>0$  and satisfying $\sum\limits_{i=i}^{n-3}b_{i}=1$.
 Let $L(\mathbf{a},\mathbf{b})$ be the polygon as shown in Figure \ref{fig:L1}.
 The vector $\mathbf{a}=\left(a_1, \cdots, a_{n-3}\right)$ describes the heights of the cylinders $\Pi_1, \cdots, \Pi_{n-3}$ while the vector
 $\mathbf{b}=\left(b_1, \cdots, b_{n-3}\right)$ measures the circumferences (note that the circumference of $\Pi_i$ is equal to 
 $2(b_1+\cdots+b_i)$). 
 There are $n$ vertices, denoted by
 $P_k, k=1, \cdots,n$, at which the interior angles are $\frac{\pi}{2}$. The remaining vertices are denote by $Q_j,j=1,2,\cdots,n-4$, at which the interior angles are $\frac{3\pi}{2}$.

 Let $S(\mathbf{a},\mathbf{b})$ be the double of  $L(\mathbf{a},\mathbf{b})$,
  which is considered as a Riemann sphere with  marked points $P_1, \cdots, P_n$.
Let $\phi(\mathbf{a},\mathbf{b})$ be the quadratic differential on $S(\mathbf{a},\mathbf{b})$ that is obtained by 
gluing the two copies of $dz^2$ on $L(\mathbf{a},\mathbf{b})$. 
It is obvious that $\phi(\mathbf{a},\mathbf{b})$ is a Jenkins-Strebel differential,
 with $(n-4)$ simple zeros corresponding to $Q_j,j=1,2,\cdots,n-4$. The surface $S(\mathbf{a},\mathbf{b})$ is decomposed by the critical graph of $\phi(\mathbf{a},\mathbf{b})$ into $(n-3)$ annuli $\Pi_1,\Pi_2,\cdots,
\Pi_{n-3}$, each of which is swept out by closed horizontal trajectories.
We call  $\phi(\mathbf{a},\mathbf{b})$ a staircase surface.

\begin{figure}
    \centering
    \begin{tikzpicture}[line width=1pt]
        \draw (-0.7,0)--(-0.7,1.2)--(-2.3,1.2)--(-2.3,2.4)--(-3.9,2.4)--(-3.9,3.6)--(-5.7,3.6)--(-5.7,0);
        \draw[dashed] (-5.7,0)--(-0.7,0);
        \draw (-0.7,0)--(1,0)--(1,-1.3)--(-5.7,-1.3)--(-5.7,0);
        \node at (1.2,-1.5){$P_1$};
        \node at (1.2,0){$P_2$};
        \node at (-2.15,2.6){$P_{n-3}$};
        \node at (-3.75,3.75){$P_{n-2}$};
        \node at (-5.9,3.77){$P_{n-1}$};
        \node at (-5.9,-1.5){$P_{n}$};
        \node at (-1.8,1.4){$Q_{n-5}$};
        \node at (-3.45,2.6){$Q_{n-4}$};
        \draw[dashed] (-5.7,1.2)--(-2.3,1.2);
        \draw[dashed] (-5.7,2.4)--(-3.9,2.4);
        \node at (-2.2,-1.5){1};
        \node at (-1.5,1){$\cdots$};
        \node at (-3,2.2){$b_{2}$};
        \node at (-4.7,3.8){$b_{1}$};
        \node at (-3.2,0.6){$\cdots$};
        \node at (-3.6,1.8){$\Pi_2$};
        \node at (-4.7,3){$\Pi_1$};
        \node at (-2.7,-0.6){$\Pi_{n-3}$};
        \node at (-6.1,0.6){$\cdots$};
        \node at (-6.1,1.8){$a_{2}$};
        \node at (-6.1,3){$a_{1}$};
        \node at (-6.2,-0.6){$a_{n-3}$};
        \node at (-0.7,-0.25){$Q_1$};
        \node at (0.3,-0.25){$b_{n-3}$};
    \end{tikzpicture}
    \caption{The polygon $L(\mathbf{a},\mathbf{b})$.}
    \label{fig:L1}
\end{figure}


Fix $\phi=\phi(\mathbf{a},\mathbf{b})$ and denote $S= S(\mathbf{a},\mathbf{b})$.
Theorem \ref{thm:staircase} can be restated as follows.
 \begin{Theorem}\label{thm:typeA}
 The Teichm\"uller disk 
 $\tau^{\phi} : \mathbb{H} \to \Tei_{0,n}$ is not a holomorphic retraction.  
 \end{Theorem}

We shall prove Theorem \ref{thm:typeA} by contradiction. 

\begin{Lemma}\label{Lemma 4.2} Assume that the Teichm\"uller disk 
 $\tau^{\phi} : \mathbb{H} \to \Tei_{0,n}$ is a holomorphic retract of $\Tei_{0,n}$. Then
there exists a holomorphic function $F$ on $\Tei_{0,n}$ such that $$F\left(S(\mathbf{x},\mathbf{b})\right)= \sum\limits_{k=1}^{n-3}\bigg(x_k \sum\limits_{j=1}^{k}b_j \bigg) $$ for every $\mathbf{x}\in \mathbb{R}_{\geq 0}^{n-3}$ with $x_{n-3}>0$. 
\end{Lemma}

\begin{Remark}
    We allow $x_1, \cdots, x_{n-4}$
    to be zero, which indicates that the cylinders  $\Pi_1, \cdots, \Pi_{n-4}$ could degenerate. 
\end{Remark}

\begin{proof}

Suppose that there is a holomorphic retraction $\Phi : \Tei_{0,n} \to \mathbb{H}$ of the Teichm\"uller disk $\tau^{\phi}$. 
Define $$\operatorname{Area}(\Pi_k)=\frac{1}{{\|\phi\|}}\int_{\Pi_k}|\phi|
 =\frac{a_k \bigg(\sum\limits_{j=1}^k b_j \bigg)}{\sum\limits_{i=1}^{n-3}\bigg(a_i \sum\limits_{j=1}^{i}b_j \bigg)} \quad (k=1,2,\cdots,n-3).$$
 Let $$\mathcal{E}:\mathbb{H}^{n-3}\rightarrow\Tei_{0,n}$$ be the Teichm\"uller polydisk associated to $\phi$. 
By Theorem \ref{thm:GM-JS}, we can assume that $\Phi$ satisfies  $$\left(\Phi\circ \mathcal{E}\right)(\lambda_1,\lambda_2,\cdots,\lambda_{n-3})=\sum\limits_{k=1}^{n-3} \operatorname{Area}(\Pi_k) \lambda_k$$
for all $(\lambda_1,\lambda_2,\cdots,\lambda_{n-3})\in \mathbb{H}^{n-3}$.

Note that for any $\mathbf{x}=(x_1, \cdots, x_{n-3})\in \mathbb{R}_{+}^{n-3}$, we have $$\mathcal{E}\left(\frac{x_1}{a_1}i,\frac{x_2}{a_2} i,\cdots,\frac{x_{n-3}}{a_{n-3} }i\right)=S(\mathbf{x},\mathbf{b}).$$ 
Thus we have the following equality \begin{equation}\label{equ:linear}
 \Phi(S(\mathbf{x},\mathbf{b}))=
 \frac{\sum\limits_{k=1}^{n-3}\bigg(x_k
 \sum\limits_{j=1}^{k}b_j\bigg)}{\sum\limits_{k=1}^{n-3}
 \bigg(a_k\sum\limits_{j=1}^{k}b_j \bigg)}i.
\end{equation}

We take 
 $$F=\left( \sum\limits_{k=1}^{n-3}\bigg(a_k \sum\limits_{j=1}^{k}b_j \bigg)\right) \frac{\Phi}{i},$$ which satisfies
 $$F\left(S(\mathbf{x},\mathbf{b})\right)= \sum\limits_{k=1}^{n-3}\bigg(x_k \sum\limits_{j=1}^{k}b_j \bigg) $$ for every $\mathbf{x} \in \mathbb{R}_{+}^{n-3}$. 
By continuity, the above equality holds for every $\mathbf{x} \in \mathbb{R}_{\geq 0}^{n-3}$ satisfying $x_{n-3}> 0$.
\end{proof}

It is obvious that $F\left(S(\mathbf{x},\mathbf{b})\right)$ is the area 
of the polygon  $\L(\mathbf{x}, \mathbf{b})$. We shall apply the Schwarz-Christoffel formula to show that  $F\left(S(\mathbf{x},\mathbf{b})\right)$  is not smooth (see Proposition \ref{prop:area} for the precise result). Then we will reach a contradiction to 
the fact that $F$ is  holomorphic.

\subsection{The Schwarz-Christoffel maps.}\label{SCM}
Let $L(\mathbf{a},\mathbf{b})$ be a staircase polygon. We allow $a_k, k\neq n-3$ to be $0$. 
Let $f:\mathbb{H} \to L(\mathbf{a},\mathbf{b})$ be the Riemann mapping such that $f(\infty)=P_1, f(-1)=P_2, f(1)=P_n$. Assume that $-1<\eta_1\leq\xi_1<\eta_2\leq\xi_2<\cdots<\eta_{n-4}\leq\xi_{n-4}<\xi_{n-3}<1$
are points on the real line such that 
$f(\xi_k)=P_{k+2} \ (k=1,2,\cdots,n-3)$ 
and
$f(\eta_k)=Q_k \ (k=1,2,\cdots,n-4)$. 
By  the Schwarz-Christoffel formula, $f$ is given by $$f(z)=J\int_1^z\frac{\prod\limits_{k=1}^{n-4}\sqrt{\zeta-\eta_k}}{\prod\limits_{k=1}^{n-3}\sqrt{\zeta-\xi_k}\sqrt{\zeta-1}\sqrt{\zeta+1}}d\zeta.$$
Since $|P_nP_1|=1$, we have  $$\frac{1}{J}=\int_1^{\infty}\frac{\prod\limits_{k=1}^{n-4}\sqrt{x-\eta_k}}{\prod\limits_{k=1}^{n-3}\sqrt{x-\xi_k}\sqrt{x-1}\sqrt{x+1}}dx.$$
As the punctured sphere $S(\mathbf{a},\mathbf{b})$ is the double of the polygon $L(\mathbf{a},\mathbf{b})$, we can extend $f$ to a conformal mapping $$f: \mathbb{C}\setminus
\{-1,\xi_1,\xi_2,\cdots,\xi_{n-3},1\}\rightarrow S(\mathbf{a},\mathbf{b}).$$ 

\subsection*{Computing the Schwarz-Christoffel maps.}

Consider the case that  $L(\mathbf{a},\mathbf{b})$ is a  rectangle with $n-4$ extra marked points $P_3, P_4,\cdots, P_{n-2}$ on the top side. Denote such a rectangle by $L(\mathbf{a}_0,\mathbf{b})$, where $a_0=(0,\cdots,0,a_{n-3})$. So in the Schwarz-Christoffel map corresponding to $L(\mathbf{a}_0,\mathbf{b})$, we have  $-1<\eta_1=\xi_1<\eta_2=\xi_2<\cdots<\eta_{n-4}=\xi_{n-4}<\xi_{n-3}<1$.

Now we consider $\mathbf{s}=(s_1,\cdots, s_{n-4})$, where $s_k \ (k=1,2,\cdots,n-4)$ are sufficiently small positive numbers.
For each $\mathbf{s}$, we deform $L(\mathbf{a}_0,\mathbf{b})$ into a new staircase polygon $L(\mathbf{s})$
in a way that
the Schwarz-Christoffel map  
sends $(\xi_1, \xi_2,\cdots, \xi_{n-3})$ into $(P_3, P_4,\cdots, P_{n-1})$ and $(\xi_1-s_1, \xi_2-s_2,\cdots,\xi_{n-4}-s_{n-4})$ into $(Q_1, Q_2,\cdots,Q_{n-4})$.  
In other words, we have $\eta_k=\xi_k-s_k$ for $k=1,2,\cdots,n-4$. 
Let $\widehat{L}(\mathbf{s})$ be the double of $L(\mathbf{s})$, with the natural Jenkins-Strebel differential
$\phi(\mathbf{s})$
induced by the flat metric on $ L(\mathbf{s})$.  
Note that the Riemann surface structure induced by $\widehat{L}(\mathbf{s})$ is always equivalent to  $\mathbb{C}\setminus
\{-1,\xi_1,\xi_2,\cdots,\xi_{n-3},1\}$.
So we obtain a  (n-4)-parameter family of Jenkins-Strebel differentials 
which represent the same point in $\Tei_{0,n}$. 

To compute $\mathbf{a}(\mathbf{s})$ and $\mathbf{b}(\mathbf{s})$, we set

   $$ A_{j}(\mathbf{s})=
   -i\int_{\xi_{n-3-j}-s_{n-3-j}}^{\xi_{n-3-j}}
   \frac{\prod\limits_{k=1}^{n-4}\sqrt{x-\xi_k+s_k}}
   {\prod\limits_{k=1}^{n-3}\sqrt{x-\xi_k}\sqrt{x-1}\sqrt{x+1}}dx$$
   for $j=1,2,\cdots,n-4$, and 
    $$A_{n-3}(\mathbf{s})=-i\int_{-\infty}^{-1}
    \frac{\prod\limits_{k=1}^{n-4}\sqrt{x-\xi_k+s_k}}{\prod
    \limits_{k=1}^{n-3}\sqrt{x-\xi_k}\sqrt{x-1}\sqrt{x+1}}dx.$$
And we set
    $$B_j(\mathbf{s})=-\int_{\xi_{n-3-k}}^{\xi_{n-2-k}-s_{n-2-k}}
    \frac{\prod\limits_{k=1}^{n-4}\sqrt{x-\xi_k+s_k}}
    {\prod\limits_{k=1}^{n-3}\sqrt{x-\xi_k}\sqrt{x-1}\sqrt{x+1}}dx$$
    for $j=i,\cdots,n-4$, and, in particular,
    $$B_1(\mathbf{s})=-\int_{\xi_{n-4}}^{\xi_{n-3}}
    \frac{\prod\limits_{k=1}^{n-4}\sqrt{x-\xi_k+s_k}}
    {\prod\limits_{k=1}^{n-3}\sqrt{x-\xi_k}\sqrt{x-1}\sqrt{x+1}}dx.$$
    Denote 
    $$\frac{1}{J(\mathbf{s})}=
   \int_{1}^{\infty}\frac{\prod\limits_{k=1}^{n-4}\sqrt{x-\xi_k+s_k}}{\prod\limits_{k=1}^{n-3}\sqrt{x-\xi_k}\sqrt{x-1}\sqrt{x+1}}dx.
    $$




It follows by definition that
 $$ a_{j}(\mathbf{s}) = {J(\mathbf{s})}{A_{j}(\mathbf{s})}, \quad b_j(\mathbf{s})=J(\mathbf{s})B_j(\mathbf{s}).$$
  

Both $J(\mathbf{s})$ and $A_{n-3}(\mathbf{s})$ are real analytic with respect to $\mathbf{s}$,
 since their integrands depend analytically on $s_1,s_2,\cdots,s_{n-4}$ and are uniformly continuous.
 
 \subsection*{Convention} We adopt the following convention. 
 Denote $|\mathbf{s}|=\sqrt{s_1^2+\cdots+s_{n-4}^2}$. If two functions $h(\mathbf{s})$ and $g(\mathbf{s})$ satisfy $\frac{h(\mathbf{s})}{g(\mathbf{s})} \to 0$ as $|\mathbf{s}|\to 0$,
 then we denote $h(\mathbf{s})=o\left(g(\mathbf{s})\right)$.
 If there is a constant $C$ such that $\left|\frac{h(\mathbf{s})}{g(\mathbf{s})}\right| < C$, then
we denote $h(\mathbf{s})=O\left(g(\mathbf{s})\right)$.
We say that $\mathbf{s}$ is sufficiently small if $|\mathbf{s}|$ is sufficiently small. In the following, 
we denote $$h(x)=\frac{1}{\sqrt{1-x}\sqrt{x+1}\sqrt{\xi_{n-3}-x}}.$$
 
\begin{Proposition}\label{prop:B-C}
With the about notation, for each $k=1,2,\cdots,n-4$, we have for $\mathbf{s}$ sufficiently small
    $$A_{k}(\mathbf{s})=\frac{\pi}{2} h(\xi_{n-3-k}) s_{n-3-k}+O(|\mathbf{s}|)  s_{n-3-k}.$$
\end{Proposition}

\begin{proof}
For $s_j>0$ sufficiently small, we have  for  $x\in[\xi_{n-3-k}-s_{n-3-k},\xi_{n-3-k}]$,
\begin{eqnarray*}
 \frac{-i}{\sqrt{x-1}\sqrt{x+1}\sqrt{x-\xi_{n-3}}} &=& \frac{i}{\sqrt{1-\xi_{n-3-k}}\sqrt{1+\xi_{n-3-k}}\sqrt{\xi_{n-3}-\xi_{n-3-k}}}+ O(s_{n-3-k}), \\
 \prod\limits_{ j\neq n-3-k}\frac{\sqrt{x-\xi_j+s_j}}{\sqrt{x-\xi_j}} &=&\prod\limits_{ j\neq n-3-k}\left( 1+ O(s_j)\right).
\end{eqnarray*}

 So we can write
\begin{eqnarray*}
 A_{k}(\mathbf{s}) &=& \left(ih(\xi_{n-3-k})+O(s_{n-3-k}) \right)\int_{\xi_{n-3-k}-s_{n-3-k}}^{\xi_{n-3-k}}\prod\limits_{j=1}^{n-4}\frac{\sqrt{x-\xi_j+s_j}}{\sqrt{x-\xi_j}}dx \\
  &=& \left(ih(\xi_{n-3-k})+O(s_{n-3-k}) \right)\prod\limits_{j\neq n-3-k}\left( 1+ O(s_j)\right)
  \int_{\xi_{n-3-k}-s_{n-3-k}}^{\xi_{n-3-k}}\frac{\sqrt{x-\xi_{n-3-k}+s_{n-3-k}}}{\sqrt{x-\xi_{n-3-k}}}dx,
\end{eqnarray*}

It remains to compute the integral 
$$\int_{\xi_{n-3-k}-s_{n-3-k}}^{\xi_{n-3-k}}\frac{\sqrt{x-\xi_{n-3-k}+s_{n-3-k}}}{\sqrt{x-\xi_{n-3-k}}}dx.$$ Note that 
$$\int_{\xi_{n-3-k}-s_{n-3-k}}^{\xi_{n-3-k}}\frac{\sqrt{x-\xi_{n-3-k}+s_{n-3-k}}}{\sqrt{x-\xi_{n-3-k}}}dx=-i\int_{\xi_{n-3-k}-s_{n-3-k}}^{\xi_{n-3-k}}\frac{\sqrt{x-\xi_{n-3-k}+s_{n-3-k}}}{\sqrt{\xi_{n-3-k}-x}}dx,$$
Substituting $\xi_{n-3-k}-x=y$ then yields 
\begin{eqnarray*}
 &&-i\int_{\xi_{n-3-k}-s_{n-3-k}}^{\xi_{n-3-k}}\frac{\sqrt{x-\xi_{n-3-k}+s_{n-3-k}}}{\sqrt{\xi_{n-3-k}-x}}dx \\
 &=& -i\int_0^{s_{n-3-k}}\frac{\sqrt{s_{n-3-k}-y}}{\sqrt{y}}dy\\&=&-i\left[ \sqrt{s_{n-3-k}-y}\sqrt{y}+s_{n-3-k}\arctan\left( \frac{\sqrt{y}}{\sqrt{s_{n-3-k}-y}}\right)\right]\bigg|_0^{s_{n-3-k}} \\
 &=& -i\frac{\pi }{2} s_{n-3-k}. 
\end{eqnarray*}
We conclude that 
\begin{eqnarray*}
   A_{k}(\mathbf{s})=\frac{\pi}{2} h(\xi_{n-3-k}) s_{n-3-k}+O(|\mathbf{s}|)s_{n-3-k}.
\end{eqnarray*}
\end{proof}

Recall that 
  
$$B_k(\mathbf{0}) = -\int_{\xi_{n-3-k}}^{\xi_{n-2-k}}\frac{1}{\sqrt{x-1}\sqrt{x+1}\sqrt{x-\xi_3}} dx
\quad (k=1,2,\cdots,n-4).   $$
\begin{Proposition}\label{prop:P}
With the about notation, we have for $\mathbf{s}$ sufficiently small
\begin{eqnarray*}
        B_{1}(\mathbf{s})-B_{1}(\mathbf{0})&=&\frac{1}{2} h(\xi_{n-4})  s_{n-4}\ln \frac{1}{s_{n-4}}
    +o\left( s_{n-4}\ln \frac{1}{s_{n-4}} \right) +O(|\mathbf{s}|).
\end{eqnarray*}
\end{Proposition}

\begin{proof}
By definition, 
\begin{align*}
B_{1}(\mathbf{s})-B_{1}(\mathbf{0})  = \int_{\xi_{n-4}}^{\xi_{n-3}}h(x)\left(\prod\limits_{j=1}^{n-4}\frac{\sqrt{x-\xi_j+s_j}}{\sqrt{x-\xi_j}}-1\right) dx.\\
\end{align*}

Let $\epsilon=\epsilon(s_{n-4})$ be a function to  be explicitly determined later such that 
$$\epsilon(s_{n-4}) \to 0 \quad \text{and} \quad \frac{s_{n-4}}{\epsilon(s_{n-4})} \to 0$$ as $s_{n-4}\to 0$.
Write the above integral as the sum
\begin{eqnarray*}
I_1+ I_2 &=& \int_{\xi_{n-4}}^{\xi_{n-4}+\epsilon}h(x)
\left(\prod\limits_{j=1}^{n-4}\frac{\sqrt{x-\xi_j+s_j}}{\sqrt{x-\xi_j}}-1\right)dx \\
&& + \int_{\xi_{n-4}+\epsilon}^{\xi_{n-3}}h(x)
\left(\prod\limits_{j=1}^{n-4}\frac{\sqrt{x-\xi_j+s_j}}{\sqrt{x-\xi_j}}-1\right)dx . \\
\end{eqnarray*}

For the first term $I_1$, note that 
\begin{eqnarray}\label{equ:P}
h(x)
=\frac{1}{\sqrt{1-\xi_{n-4}}\sqrt{1+\xi_{n-4}}\sqrt{\xi_{n-3}-\xi_{n-4}}}+O(\epsilon).
\end{eqnarray} 
for $\xi_{n-4}\leq x\leq \xi_{n-4}+\epsilon$. On the other hand, 
\begin{eqnarray}\label{equ:P0}
  && \int_{\xi_{n-4}}^{\xi_{n-4}+\epsilon}\left(\prod\limits_{j=1}^{n-4}\frac{\sqrt{x-\xi_j+s_j}}{\sqrt{x-\xi_j}}-1\right)dx \nonumber \\
  &=& \prod_{j=1}^{n-5}\left(1+O(s_j)\right)
\int_{\xi_{n-4}}^{\xi_{n-4}+\epsilon}\frac{\sqrt{x-\xi_{n-4}+s_{n-4}}}{\sqrt {x-\xi_{n-4}}} dx -\epsilon.
\end{eqnarray}

Substituting $x-\xi_{n-4}=y$ yields

\begin{eqnarray}\label{equ:P1}
    && \int_{\xi_{n-4}}^{\xi_{n-4}+\epsilon}\frac{\sqrt{x-\xi_{n-4}+s_{n-4}}}{\sqrt {x-\xi_{n-4}}}dx  \nonumber  \\ 
    &=&\int_0^\epsilon\frac{\sqrt{y+s_{n-4}}}{\sqrt{y}}dy \nonumber  \\
    &=& \left( \sqrt{y}\sqrt{y+s_{n-4}}+s_{n-4}\ln (\sqrt{y}+\sqrt{y+s_{n-4}})\right)\mid_0^\epsilon \nonumber  \\
    &=&\sqrt{\epsilon}\sqrt{\epsilon+s_{n-4}}+s_{n-4}\ln \frac{(\sqrt{\epsilon}+\sqrt{\epsilon+s_{n-4}})}{\sqrt{s_{n-4}}}.  
    \end{eqnarray}

Combining \eqref{equ:P},\eqref{equ:P0} and \eqref{equ:P1},  the first term satisfies
\begin{eqnarray*}
 I_1 &=& \int_{\xi_{n-4}}^{\xi_{n-4}+\epsilon}h(x)
\left(\prod\limits_{j=1}^{n-4}\frac{\sqrt{x-\xi_j+s_j}}{\sqrt{x-\xi_j}}-1\right)dx  \nonumber  \\
&=&\left(h(\xi_{n-4})+O(\epsilon)\right)\left( \sqrt{\epsilon}\sqrt{\epsilon+s_{n-4}}-\epsilon+s_{n-4}\ln \frac{(\sqrt{\epsilon}+\sqrt{\epsilon+s_{n-4}})}{\sqrt{s_{n-4}}} \right) \nonumber \\
&&+ \sum\limits_{j=1}^{n-5}O(s_j)\left( \sqrt{\epsilon}\sqrt{\epsilon+s_{n-4}} +s_{n-4}\ln \frac{(\sqrt{\epsilon}+\sqrt{\epsilon+s_{n-4}})}{\sqrt{s_{n-4}}} \right).
\end{eqnarray*}

Note that 
\begin{eqnarray*}
 \sqrt{\epsilon}\sqrt{\epsilon+s_{n-4}} &=& \epsilon + o(\epsilon), \\
  \sqrt{\epsilon}\sqrt{\epsilon+s_{n-4}}-\epsilon &=& \frac{s_{n-4}}{1+\sqrt{1+s_{n-4}/\epsilon}}= \frac{s_{n-4}}{2}+o(s_{n-4}),\\
  s_{n-4}\ln \frac{(\sqrt{\epsilon}+\sqrt{\epsilon+s_{n-4}})}{\sqrt{s_{n-4}}} &=&s_{n-4}\ln\left(\frac{\sqrt{\epsilon}}{\sqrt{s_{n-4}}}\left(1+\sqrt{1+\frac{s_{n-4}}{\epsilon}}\right)\right)\\&=& \frac{1}{2}s_{n-4}\ln \frac{1}{s_{n-4}} - \frac{1}{2}s_{n-4}\ln \frac{1}{\epsilon} + O\left(\frac{s_{n-4}^2}{\epsilon}\right)+s_{n-4}\ln 2 \\
\end{eqnarray*}
Therefore, 
\begin{eqnarray}\label{equ:P2}
I_1&=&\frac{1}{2}h(\xi_{n-4}) s_{n-4}\left(\ln \frac{1}{s_{n-4}}-\ln \frac{1}{\epsilon}\right)+ o\left(s_{n-4}\ln \frac{1}{s_{n-4}}\right)+\sum\limits_{j=1}^{n-5}O(s_j).
\end{eqnarray}

For the second term $I_2$, we note that when $x\geq \xi_{n-4}+\epsilon$, 
\begin{eqnarray*}
 && \prod\limits_{j=1}^{n-4}\frac{\sqrt{x-\xi_j+s_j}}{\sqrt{x-\xi_j}} -1 \\
 &=& \prod_{j=1}^{n-5}\left(1+O(s_j)\right)\left( 1+ O\left(\frac{s_{n-4}}{\epsilon}\right)\right) -1 \\
 &=&\sum\limits_{j=1}^{n-5}O(s_j)+ O \left( \frac{s_{n-4}}{\epsilon}\right).
\end{eqnarray*} 
Thus
\begin{eqnarray}\label{equ:P3}
I_2&=& \int^{\xi_{n-3}}_{\xi_{n-4}+\epsilon}h(x)
\left(\prod\limits_{j=1}^{n-4}\frac{\sqrt{x-\xi_j+s_j}}{\sqrt{x-\xi_j}}-1\right)dx \nonumber \\
&=& \sum\limits_{j=1}^{n-5}O(s_j)+ O \left( \frac{s_{n-4}}{\epsilon}\right).
\end{eqnarray}

To conclude our result, we take $$\epsilon=\frac{1}{\sqrt{\ln \frac{1}{s_{n-4}}}}.$$
Then by \eqref{equ:P2} and \eqref{equ:P3}, we have

 \begin{eqnarray*}
     B_{1}(\mathbf{s})-B_{1}(\mathbf{0})&=&\frac{1}{2} h(\xi_{n-4})\  s_{n-4}\ln \frac{1}{s_{n-4}}
    +o\left( s_{n-4}\ln \frac{1}{s_{n-4}}\right) +O(|\mathbf{s}|).
 \end{eqnarray*}
 \end{proof}

\begin{Proposition}\label{prop:Q}
    With the about notation, we have  for each $k=2,3,\cdots,n-4$ and $\mathbf{s}$ sufficiently small
\begin{eqnarray*}
 &&B_{k}(\mathbf{s})-B_{k}(\mathbf{0}) \\
 &=& \frac{-1}{2} h(\xi_{n-2-k}) s_{n-2-k} \ln \frac{1}{s_{n-2-k}}+\frac{1}{2}h(\xi_{n-3-k}) s_{n-3-k}\ln \frac{1}{s_{n-3-k}}\\
  &&+o\left(s_{n-2-k} \ln \frac{1}{s_{n-2-k}}\right)+o\left(s_{n-3-k}\ln \frac{1}{s_{n-3-k}}\right)+ O\left(|\mathbf{s}|\right).
\end{eqnarray*}

\end{Proposition}
\begin{proof}
By the formula of $B_{k}(\mathbf{s})$, we have 

\begin{eqnarray}\label{equ:Q1}
B_{k}(\mathbf{s})-B_{k}(\mathbf{0})&=&\int_{\xi_{n-3-k}}^{\xi_{n-2-k}-s_{n-2-k}}h(x)
\left(\prod\limits_{j=1}^{n-4}\frac{\sqrt{x-\xi_j+s_j}}{\sqrt{x-\xi_j}}-1\right)dx  \nonumber \\ 
&& - \int_{\xi_{n-2-k}-s_{n-2-k}}^{\xi_{n-2-k}}h(x)dx
\end{eqnarray}
Obviously, the second term  
 $$\int_{\xi_{n-2-k}-s_{n-2-k}}^{\xi_{n-2-k}}h(x)dx=O(s_{n-2-k}).$$ 
It remains to study the first term in \eqref{equ:Q1}.

Let $N=\ln \frac{1}{s_{n-3-k}}$. We write the first term in \eqref{equ:Q1} as
\begin{eqnarray}\label{equ:Q2}
 && \int_{\xi_{n-3-k}}^{\xi_{n-2-k}-s_{n-2-k}} h(x) \left(\prod\limits_{j=1}^{n-4}\frac{\sqrt{x-\xi_j+s_j}}{\sqrt{x-\xi_j}}-1\right)dx  \\
&=&  I_3+I_4 \nonumber \\
&=& \int_{\xi_{n-3-k}}^{\xi_{n-3-k}+Ns_{n-3-k}} h(x)\left(\prod\limits_{j=1}^{n-4}\frac{\sqrt{x-\xi_j+s_j}}{\sqrt{x-\xi_j}}-1\right)dx \nonumber \\ &&+\int_{\xi_{n-3-k}+Ns_{n-3-k}}^{\xi_{n-2-k}-s_{n-2-k}} h(x)\left(\prod\limits_{j=1}^{n-4}\frac{\sqrt{x-\xi_j+s_j}}{\sqrt{x-\xi_j}}-1\right)dx. \nonumber
\end{eqnarray}

For $I_3$ in \eqref{equ:Q2}, we have 
\begin{eqnarray*}
  I_3 &=& \int_{\xi_{n-3-k}}^{\xi_{n-3-k}+Ns_{n-3-k}} h(x)\left(\prod\limits_{j=1}^{n-4}\frac{\sqrt{x-\xi_j+s_j}}{\sqrt{x-\xi_j}}-1\right)dx
   \\
   &=&  \int_{\xi_{n-3-k}}^{\xi_{n-3-k}+Ns_{n-3-k}} h(x) \left(\prod\limits_{j\neq n-3-k}\frac{\sqrt{x-\xi_j+s_j}}{\sqrt{x-\xi_j}}-1\right)dx \\
   && +\int_{\xi_{n-3-k}}^{\xi_{n-3-k}+Ns_{n-3-k}} h(x)\prod\limits_{j\neq n-3-k}\frac{\sqrt{x-\xi_j+s_j}}{\sqrt{x-\xi_j}}\left(\frac{\sqrt{x-\xi_{n-3-k}+s_{n-3-k}}}{\sqrt{x-\xi_{n-3-k}}}-1\right)dx \\
   &=& \sum\limits_{j\neq n-3-k}O(s_j) \left( s_{n-3-k} \ln \frac{1}{s_{n-3-k}}\right) \\&&+ \prod\limits_{j\neq n-3-k}\left(1+O(s_j)\right)\int_{\xi_{n-3-k}}^{\xi_{n-3-k}+Ns_{n-3-k}} h(x)\left(\frac{\sqrt{x-\xi_{n-3-k}+s_{n-3-k}}}{\sqrt{x-\xi_{n-3-k}}}-1\right)dx.
\end{eqnarray*}
Substituting $x-\xi_{n-3-k}=y$, we have 
\begin{eqnarray*}
&& \int_{\xi_{n-3-k}}^{\xi_{n-3-k}+Ns_{n-3-k}}h(x) \left(\frac{\sqrt{x-\xi_{n-3-k}+s_{n-3-k}}}{\sqrt{x-\xi_{n-3-k}}}-1\right)dx \\
&=& \left( h(\xi_{n-3-k})+O\left(s_{n-3-k}\ln\frac{1}{s_{n-3-k}}\right) \right) \left( \int_0^{Ns_{n-3-k}}\frac{\sqrt{y+s_{n-3-k}}}{\sqrt{y}}dy - Ns_{n-3-k}\right)\\
&=&  \left( h_1(\xi_{n-3-k})+O\left(s_{n-3-k}\ln\frac{1}{s_{n-3-k}}\right) \right) \left(\sqrt{N(N+1)}-N+\ln(\sqrt{N}+\sqrt{N+1})\right)s_{n-3-k}.
\end{eqnarray*}
Since $\sqrt{N(N+1)}-N$ is bounded and $\ln(\sqrt{N}+\sqrt{N+1})$ is approximately 
$\frac{1}{2} \ln \ln \frac{1}{s_{n-3-k}}$,
we conclude that  $$I_3= o\left(s_{n-3-k} \ln \frac{1}{s_{n-3-k}}\right)+ \sum\limits_{j\neq n-3-k} O(s_j)\left( s_{n-3-k} \ln \frac{1}{s_{n-3-k}}\right).$$

For $I_4$ in \eqref{equ:Q2}, we write
\begin{eqnarray*}
   I_4 &=& \int_{\xi_{n-3-k}+Ns_{n-3-k}}^{\xi_{n-2-k}-s_{n-2-k}} h(x)\left(\prod\limits_{j=1}^{n-4}\frac{\sqrt{x-\xi_j+s_j}}{\sqrt{x-\xi_j}}-1\right)dx \\
    &=& I_5+I_6, \\
    \end{eqnarray*}
    where
  \begin{eqnarray*} 
  I_5 &=& \int_{\xi_{n-3-k}+Ns_{n-3-k}}^{\xi_{n-2-k}-s_{n-2-k}}h(x)\left(\prod\limits_{j\neq n-2-k}\frac{\sqrt{x-\xi_j+s_j}}{\sqrt{x-\xi_j}}-1\right)dx, \\
  I_6 &=& \int_{\xi_{n-3-k}+Ns_{n-3-k}}^{\xi_{n-2-k}-s_{n-2-k}}  h(x) \prod\limits_{j\neq n-2-k}\frac{\sqrt{x-\xi_j+s_j}}{\sqrt{x-\xi_j}} \left(\frac{\sqrt{x-\xi_{n-2-k}+s_{n-2-k}}}{\sqrt{x-\xi_{n-2-k}}}-1\right)dx.
\end{eqnarray*}
Substituting $x-\xi_{n-3-k}=y$, we have 
\begin{eqnarray*}
I_5 &=& \int_{\xi_{n-3-k}+Ns_{n-3-k}}^{\xi_{n-2-k}-s_{n-2-k}}h(x) \left(\prod\limits_{j\neq n-2-k}\frac{\sqrt{x-\xi_j+s_j}}{\sqrt{x-\xi_j}}-1\right)dx \\
&=& \int_{Ns_{n-3-k}}^{\xi_{n-2-k}-\xi_{n-3-k}-s_{n-2-k}} h(\xi_{n-3-k}+y) \prod\limits_{j\neq n-2-k}\left(\sqrt{1+\frac{s_j}{y+\xi_{n-3-k}-\xi_j}}-1\right) dy\\
&=& \frac{1}{2} \int_{Ns_{n-3-k}}^{\xi_{n-2-k}-\xi_{n-3-k}-s_{n-2-k}} h(\xi_{n-3-k}+y) \left( \frac{s_{n-3-k}}{y}  +O\left(\frac{s_{n-3-k}^2}{y^2}\right)+\sum\limits_{j=1,j\neq n-3-k}^{n-4}O(s_j)\right) dy\\
&=& \frac{h(\xi_{n-3-k})}{2}s_{n-3-k}\ln \frac{1}{s_{n-3-k}}+o\left(s_{n-3-k}\ln \frac{1}{s_{n-3-k}}\right)+\sum\limits_{j\neq n-3-k}O(s_j).
\end{eqnarray*}
where the last equality follows with the aid of an integration by parts.

It remains to estimate 
$$ I_6 = \int_{\xi_{n-3-k}+Ns_{n-3-k}}^{\xi_{n-2-k}-s_{n-2-k}}  h(x) \prod\limits_{j\neq n-2-k}\frac{\sqrt{x-\xi_j+s_j}}{\sqrt{x-\xi_j}} \left(\frac{\sqrt{x-\xi_{n-2-k}+s_{n-2-k}}}{\sqrt{x-\xi_{n-2-k}}}-1\right)dx.$$
Note that  $\frac{\sqrt{x-\xi_j+s_j}}{\sqrt{x-\xi_j}}>0$
when  $\xi_{n-3-k}+Ns_{n-3-k}<x<\xi_{n-2-k}-s_{n-2-k}$.

Substituting $\xi_{n-2-k}-x=y$, we have 

\begin{eqnarray*}
&& \int_{\xi_{n-3-k}+Ns_{n-3-k}}^{\xi_{n-2-k}-s_{n-2-k}}  h(x) \prod\limits_{j\neq n-2-k}\frac{\sqrt{x-\xi_j+s_j}}{\sqrt{x-\xi_j}} \left(\frac{\sqrt{x-\xi_{n-2-k}+s_{n-2-k}}}{\sqrt{x-\xi_{n-2-k}}}-1\right)dx \\
&=&  \int_{s_{n-2-k}}^{\xi_{n-2-k}-\xi_{n-3-k}-Ns_{n-3-k}}  h(\xi_{n-2-k}-y) \prod_{j\neq n-2-k}\sqrt{1+\frac{s_j}{\xi_{n-2-k}-\xi_j-y}}
 \left(\sqrt{1- \frac{s_{n-2-k}}{y}}-1\right)dy \\
 &=& 
 \big(1+O(\frac{1}{N})\big)\prod_{\substack{j\neq n-3-k\\j\neq n-2-k}}\left(1+O(s_j)\right) 
 \int_{s_{n-2-k}}^{\xi_{n-2-k}-\xi_{n-3-k}-Ns_{n-3-k}}  h(\xi_{n-2-k}-y)
 \left( \sqrt{1- \frac{s_{n-2-k}}{y}}-1\right)dy, 
\end{eqnarray*}
since
$$\sqrt{1+\frac{s_{n-3-k}}{\xi_{n-2-k}-\xi_{n-3-k}-y}}=1+ O\left(\frac{1}{N}\right).$$

Now we take $M=\ln \frac{1}{s_{n-2-k}}$ and write 
\begin{eqnarray*}
&& \int_{s_{n-2-k}}^{\xi_{n-2-k}-\xi_{n-3-k}-Ns_{n-3-k}}  h(\xi_{n-2-k}-y)
 \left(\sqrt{1- \frac{s_{n-2-k}}{y}}-1\right)dy \\
 & =& \int_{s_{n-2-k}}^{Ms_{n-2-k}}  h(\xi_{n-2-k}-y)
 \left(\sqrt{1- \frac{s_{n-2-k}}{y}}-1\right)dy\\&&+\int_{Ms_{n-2-k}}^{\xi_{n-2-k}-\xi_{n-3-k}-Ns_{n-3-k}}  h(\xi_{n-2-k}-y)
 \left(\sqrt{1- \frac{s_{n-2-k}}{y}}-1\right)dy.
\end{eqnarray*}
Similar to the calculation of  $I_5$, one can show that the first term is $o\left(s_{n-2-k}\ln\frac{1}{s_{n-2-k}}\right)$. For the second term, we have
$$ \sqrt{1- \frac{s_{n-2-k}}{y}} -1 =  -\frac{1}{2} \frac{s_{n-2-k}}{y} + O\left(  \frac{s_{n-2-k}^2}{y^2}\right).$$
Applying integration by parts, we obtain 
\begin{eqnarray*}
&& \int_{Ms_{n-2-k}}^{\xi_{n-2-k}-\xi_{n-3-k}-Ns_{n-3-k}}  h(\xi_{n-2-k}-y)
 \left( \sqrt{1- \frac{s_{n-2-k}}{y}}-1\right)dy \\
 &=& - \frac{1}{2}h(\xi_{n-2-k}) s_{n-2-k}\ln\frac{1}{s_{n-2-k}}+ o\left(s_{n-2-k}\ln\frac{1}{s_{n-2-k}}\right).
\end{eqnarray*}
So we have 
\begin{eqnarray*}
  I_6 &=& - \frac{1}{2}h(\xi_{n-2-k}) s_{n-2-k} \ln \frac{1}{s_{n-2-k}} + o( s_{n-2-k} \ln \frac{1}{s_{n-2-k}})+O\left( \frac{1}{\ln\frac{1}{s_{n-3-k}}}\right) s_{n-2-k}\ln\frac{1}{s_{n-2-k}}\\
  &&+\sum\limits_{j\neq n-3-k, n-2-k}O(s_j)s_{n-2-k}\ln\frac{1}{s_{n-2-k}}.  
\end{eqnarray*}

We conclude that 
\begin{eqnarray*}
    I_4 &=& \frac{1}{2}h(\xi_{n-3-k})s_{n-3-k}\ln \frac{1}{s_{n-3-k}}- \frac{1}{2}h(\xi_{n-2-k}) s_{n-2-k} \ln \frac{1}{s_{n-2-k}}+ o\left(s_{n-3-k}\ln \frac{1}{s_{n-3-k}}\right)\\&&+o\left(s_{n-2-k} \ln \frac{1}{s_{n-2-k}}\right)+\sum\limits_{j\neq n-2-k,n-3-k}O(s_j)+O\left( \frac{1}{\ln\frac{1}{s_{n-3-k}}}\right) s_{n-2-k}\ln\frac{1}{s_{n-2-k}}.
\end{eqnarray*}
The proof concludes by our estimates for $I_3$ and $I_4$.
\end{proof}

With the above computation, we are able to write down an expansion
 for the area of the polygon $L(\mathbf{s})$.
Recall that
$$
 h(x)=\frac{1}{\sqrt{1-x}\sqrt{1+x}\sqrt{\xi_{n-3}-x}},
 $$
 which is always positive.
 
\begin{Proposition}\label{prop:area}
    There are constants  $\beta_j$ and $\beta_{jj}>0, \ j=1,2,\cdots,n-4$  such that 
 \begin{eqnarray*}
   && \operatorname{Area}(L(\mathbf{s})) \\
    &=& a_{n-3}(\mathbf{0})+\sum\limits_{j=1}^{n-4}\beta_js_j
    +\sum\limits_{j=1}^{n-4}\beta_{jj}{s_j}^2\ln \frac{1}{s_j}
     +\sum\limits_{i,j=1}^{n-4}o(s_is_j\ln \frac{1}{s_j})
\end{eqnarray*} 
\end{Proposition}
\begin{proof}
This is a consequence of Proposition \ref{prop:B-C}, Proposition \ref{prop:P} and 
Proposition \ref{prop:Q}. In fact, we have 
\begin{eqnarray*}
J(\mathbf{s})&=&J(\mathbf{0})+\sum\limits_{j=1}^{n-4}J_is_i+\sum\limits_{i,j=1}^{n-4}O(s_is_j), \\
a_{n-3}(\mathbf{s}) &=& a_{n-3}(\mathbf{0}) + \sum\limits_{j=1}^{n-4}\alpha_is_i+\sum\limits_{i,j=1}^{n-4}O(s_is_j), \\
\end{eqnarray*}
since the two functions are real analytic. 
Moreover, by our previous computations, we have for
$k=1,2,\cdots,n-4,$
\begin{eqnarray*}
 a_{k}(\mathbf{s}) &=& J(\mathbf{s}) \left(\frac{\pi}{2}h(\xi_{n-3-k})s_{n-3-k}+\sum\limits_{j=1}^{n-4}O(s_{n-3-k}s_j)\right) \\
 &=& \frac{\pi}{2}J(\mathbf{0})h(\xi_{n-3-k}) s_{n-3-k} + \sum\limits_{j=1}^{n-4}O(s_{n-3-k}s_j), \\
 b_{1}(\mathbf{s}) &=& b_1(\mathbf{0})+\frac{1}{2}J(\mathbf{0})h(\xi_{n-4})s_{n-4}\ln \frac{1}{s_{n-4}}+\sum\limits_{j=1}^{n-4}o(s_j\ln \frac{1}{s_j}),
 \\
 \end{eqnarray*}
  and for $k=2,3,\cdots,n-4$,
 \begin{eqnarray*}
  b_{k}(\mathbf{s}) &=& b_{k}(\mathbf{0})-\frac{1}{2}J(\mathbf{0})h(\xi_{n-2-k})s_{n-2-k}\ln \frac{1}{s_{n-2-k}}+\frac{1}{2}J(\mathbf{0})h(\xi_{n-3-k})s_{n-3-k}\ln \frac{1}{s_{n-3-k}}\\&&+\sum\limits_{j=1}^{n-4}o(s_j\ln \frac{1}{s_j}).
\end{eqnarray*}
A easy computation gives 

 \begin{eqnarray*}
   && \operatorname{Area}(L(\mathbf{s})) \\
    &=& a_{n-3}(\mathbf{0})+\sum\limits_{j=1}^{n-4}\beta_js_j
    +\sum\limits_{j=1}^{n-4}\beta_{jj}{s_j}^2\ln \frac{1}{s_j}
     +\sum\limits_{i,j=1}^{n-4}o(s_is_j\ln \frac{1}{s_j}),
\end{eqnarray*} 
for $\beta_{jj}=\frac{\pi}{4}J^2(\mathbf{0})h^2(\xi_j)>0$.
\end{proof}

 \begin{Remark}
      All the constants depend smoothly on the base rectangle $L(\mathbf{a}_0, \mathbf{b})$.
 \end{Remark}

\subsection{Proof of Theorem \ref{thm:typeA}}\label{sec:proof}
Let $\phi=\phi(\mathbf{a},\mathbf{b})$ be a staircase Jenkins-Strebel differential on $S_{0,n}$. 
Assume for contradiction that the Teichm\"uller disk $\tau^{\phi}$ admits a holomorphic retraction.  Let $F$ 
be the holomorphic function on $\Tei_{0,n}$ defined by Lemma \ref{Lemma 4.2}. 

\begin{figure}[h]
    \centering
    \begin{tikzpicture}[line width=1pt]
        \draw (0,0) rectangle (9,2);
        \draw [dashed](2,0)--(2,2);
        \draw [dashed] (3.5,0)--(3.5,2);
        \node at (4.5,-0.2) {1};
        \node at (9.5,1){$a_{n-3}$};
        \node at (1,2.3){$b_1$};
        \node at (2.75,2.3){$b_2$};
        \node at (4.75,2.3){$\cdots$};
        \node at (1,1){$\Pi_{1,3}$};
        \node at (2.75,1){$\Pi_{2,3}$};
        \node at (4.75,1){$\cdots$};
        \draw[ dashed] (6.3,0)--(6.3,2);
        \node at (7.6,1) {$\prod_{n-3,3}$};
        \node at (7.6,2.5) {$1-\sum\limits_{j=1}^{n-4}b_j$};
    \end{tikzpicture}
    \caption{}
    \label{21}
\end{figure}

Consider the rectangle surface $S(\mathbf{a}_0,\mathbf{b})$,
obtained by degenerating $\Pi_1, \cdots, \Pi_{n-4}$.  
  For small $\mathbf{t}=(t_1,\cdots,t_{n-4})>\mathbf{0}$, we let $$R(\mathbf{t})=S(\mathbf{a}_0,\mathbf{b}-\mathbf{t})\in \Tei_{0,n}.$$ 
  
\begin{Lemma}\label{lemma 4.7}
The family of Riemann surfaces $\{R(\mathbf{t})\}$ is real analytic with respect to the parameters $\mathbf{t}$.
\end{Lemma}
 \begin{proof}
    We can obtain
 $R(\mathbf{t})$ by quasiconformal deformation of $R(\mathbf{0})=S(\mathbf{a}_0,\mathbf{b})$, with Beltrami differentials
$\mu(\mathbf{t})$  given by $\mu_j(\mathbf{t})$ on the annulus $\Pi_{j,3}$ (see Figure \ref{21}),
satisfying 
$$\mu_j(\mathbf{t})=\frac{-t_j}{2b_j-t_j} \quad (j=1,2,\cdots,n-4),$$
and,
$$
\mu_{n-3}(\mathbf{t})=\frac{\sum\limits_{j=1}^{n-4}t_j}{2\left(1-\sum\limits_{j=1}^{n-4}b_j\right)+\sum\limits_{j=1}^{n-4}t_j}.
$$

 It is obvious that the family of Beltrami differentials  $\{\mu(\mathbf{t})\}$ depends real-analytically on $\mathbf{t}$, so the lemma holds.


 \end{proof}

As we have shown, 
 $R(\mathbf{t})$ depends smoothly on $\mathbf{t}$, thus $F(R(\mathbf{t}))$ defines a smooth function of $\mathbf{t}$. 
In the next lemma, we show that 
$R(\mathbf{0})=S(\mathbf{a}_0,\mathbf{b})$ can be identified with some surface 
$\widehat{L}(\mathbf{s})$, which arises from 
the trivial deformation of $R(\mathbf{t})$ that are constructed in Section \ref{SCM}. 

\begin{Lemma}\label{lem:locate}
For each $\mathbf{t}$ that is sufficiently small, we can locate $\mathbf{s}$ such that the surface $\widehat{L}(\mathbf{s})$, obtained by trivial deformation of the base surface $R(\mathbf{t})$,
satisfying 
$$\mathbf{b}(\mathbf{s}) = \mathbf{b}.$$
\end{Lemma}
\begin{proof}
Let $\mathbf{t}$ be sufficiently small. Consider the surface $R(\mathbf{t})$ as $\widehat{L}(\mathbf{0})$. 
By our previous computation, there are positive constants $h(\xi_k), \ k=1,2,\cdots,n-4$ (depending smoothly on $\widehat{L}(\mathbf{0})$) such that 
\begin{eqnarray*}
  b_1(\mathbf{s})-(b_1-t_1) &=&\frac{1}{2}h(\xi_{n-4})  s_{n-4}\ln{\frac{1}{s_{n-4}}}+ \sum\limits_{j=1}^{n-4}o(s_j\ln \frac{1}{s_j}), \\
 b_{k}(\mathbf{s})- (b_k-t_k) &=& -\frac{1}{2}h(\xi_{n-2-k})s_{n-2-k}\ln \frac{1}{s_{n-2-k}}+ \frac{1}{2}h(\xi_{n-3-k}) s_{n-3-k}\ln \frac{1}{s_{n-3-k}} \\
 && +\sum\limits_{j=1}^{n-4}o(s_j\ln \frac{1}{s_j}),
\end{eqnarray*}
for $k=2,3,\cdots,n-4$.
Set
$x_{k}=s_k\ln \frac{1}{s_k}$.
So it suffices to locate $\mathbf{s}$ such that it solves the equations
\begin{eqnarray}\label{equ:solution}
  t_1 &=&\frac{1}{2}h(\xi_{n-4})  x_{n-4}+ \sum\limits_{j=1}^{n-4}o(x_j), \nonumber \\
 t_k &=& -\frac{1}{2}h(\xi_{n-2-k})x_{n-2-k}+ \frac{1}{2}h(\xi_{n-3-k}) x_{n-3-k}+\sum\limits_{j=1}^{n-4}o(x_j),
\end{eqnarray}
for $k=2,3,\cdots,n-4$.

The linear part of the above systems of equations is given by the matrix 
 $$M=\left[ 
 {\begin{array}{cccccc}
    0 & 0 & \cdots & 0 & 0 &\frac{h(\xi_{n-4})}{2} \\
    0 & 0 & \cdots & 0 & \frac{h(\xi_{n-5})}{2} & -\frac{h(\xi_{n-4})}{2}\\
    0 & 0 & \cdots & \frac{h(\xi_{n-6})}{2} & -\frac{h(\xi_{n-5})}{2} & 0\\
    \vdots & \vdots & \vdots & \vdots & 0 & 0\\
    \frac{h(\xi_1)}{2} & -\frac{h(\xi_2)}{2} & 0 & \cdots & 0 & 0
  \end{array} } \right]_{(n-4)*(n-4)}.$$ The inverse of $M$ is  
  $$\left[ {\begin{array}{cccccc}
    \frac{2}{h(\xi_1)} & \frac{2}{h(\xi_1)} & \cdots & \frac{2}{h(\xi_1)} & \frac{2}{h(\xi_1)} & \frac{2}{h(\xi_1)} \\
    \frac{2}{h(\xi_2)} & \frac{2}{h(\xi_2)} & \cdots & \frac{2}{h(\xi_2)} & \frac{2}{h(\xi_2)} & 0 \\
    \frac{2}{h(\xi_3)} & \cdots & \cdots & \frac{2}{h(\xi_3)} & 0 & 0\\
    \vdots & \vdots & \vdots & 0& 0&0\\
    \frac{2}{h(\xi_{n-4})} & 0 & \cdots & 0 &0 &0
  \end{array} } \right]_{(n-4)*(n-4)},$$ 
  which is positive. 
Thus for $\mathbf{t}$ that is sufficiently small, there exists some  $\mathbf{x}>\mathbf{0}$ (and then $\mathbf{s}>\mathbf{0}$) that solves \eqref{equ:solution}. This implies 
$ \mathbf{b}(\mathbf{s})=\mathbf{b}.$
\end{proof}

 Using the same notation as in Lemma \ref{lem:locate}, we have 
 \begin{eqnarray*}
\left[ {\begin{array}{cccccccc}
    t_1 \\
    t_2 \\
    \vdots\\
    t_{n-4}
  \end{array} } \right] = \left[ 
 {\begin{array}{cccccc}
    o(1) & o(1) & \cdots & o(1) & o(1) &\frac{h(\xi_{n-4})}{2}+o(1) \\
    o(1) & o(1) & \cdots & o(1) & \frac{h(\xi_{n-5})}{2}+o(1) & -\frac{h(\xi_{n-4})}{2}+o(1)\\
    \vdots & \vdots & \vdots & \vdots & o(1) & o(1)\\
    \frac{h(\xi_1)}{2}+o(1) & -\frac{h(\xi_2)}{2}+o(1) & o(1) & \cdots & o(1) & o(1)
  \end{array} } \right] \left[ {\begin{array}{cccccc}
    x_1 \\
    x_2 \\
    \vdots\\
    x_{n-4}
  \end{array} } \right]
\end{eqnarray*} 
   This gives 
      \begin{eqnarray*}
\left[ {\begin{array}{cccccc}
    x_1 \\
    x_2 \\
    \vdots\\
    x_{n-4}
  \end{array} } \right] = \left[ {\begin{array}{cccccc}
   \frac{2}{h(\xi_1)}+o(1) & \frac{2}{h(\xi_1)}+o(1) & \cdots & \frac{2}{h(\xi_1)}+o(1) & \frac{2}{h(\xi_1)}+o(1) & \frac{2}{h(\xi_1)}+o(1) \\\\
    \frac{2}{h(\xi_2)}+o(1) & \frac{2}{h(\xi_2)}+o(1) & \cdots & \frac{2}{h(\xi_2)}+o(1) & \frac{2}{h(\xi_2)}+o(1) & o(1) \\\\
    \frac{2}{h(\xi_3)}+o(1) & \cdots & \cdots & \frac{2}{h(\xi_3)}+o(1) & o(1) & o(1)\\\\
    \vdots & \vdots & \vdots & o(1)& o(1)&o(1)\\\\
    \frac{2}{h(\xi_{n-4})}+o(1) & o(1) & \cdots & o(1) &o(1) &o(1)
  \end{array} } \right]  \left[ {\begin{array}{cccccc}
    t_1 \\
    t_2 \\
    \vdots\\
    t_{n-4}
  \end{array} } \right]
\end{eqnarray*} 
     Recall that $x_k=s_k\ln \frac{1}{s_k}$.
     Thus $s_k=\frac{x_k}{\ln \frac{1}{s_k}}$. 
     Since
     $$\ln{\frac{1}{s_k}}=(1+o(1))\ln{\frac{1}{x_k}},$$
      We have 
       $$s_k=\frac{x_k}{(1+o(1))\ln{\frac{1}{x_k}}}.$$
    As a consequence, we obtain
    \begin{eqnarray*}
      s_k &=& \frac{\left( \frac{2}{h(\xi_k)}+o(1)\right) \sum\limits_{j=1}^{n-3-k}t_j }{\left(1+o(1)\right) \ln \left( \frac{1}{\left( \frac{2}{h(\xi_k)}+o(1)\right) \sum\limits_{j=1}^{n-3-k}t_j }\right)} 
      = \frac{\left( \frac{2}{h(\xi_k)}+o(1)\right) \sum\limits_{j=1}^{n-3-k}t_j }{\left(1+o(1)\right) \ln \left( \frac{1}{ \sum\limits_{j=1}^{n-3-k}t_j }\right)}  \\
    \end{eqnarray*}  
    The last equality holds since 
    $$
    \lim\limits_{x \to 0^+}\frac{x}{\ln \frac{1}{\left( \frac{2}{h(\xi_k)}+o(1)\right)x}}\bigg/\frac{x}{\ln\frac{1}{x}}=1.
    $$
  By Proposition \ref{prop:area}, substituting $\mathbf{s}$  with $\mathbf{t}$, we  find that  
\begin{eqnarray}\label{equ:area}
   && \operatorname{Area}(\mathbf L(\mathbf{s}))  \nonumber \\
   &= &a_{n-3}(\mathbf{0})+\sum\limits_{k=1}^{n-4}\beta_ks_k
    +\sum\limits_{k=1}^{n-4}\beta_{kk}{s_k}^2\ln \frac{1}{s_k}
     +\sum\limits_{i,j=1}^{n-4}o(s_is_j\ln \frac{1}{s_j})
    \nonumber\\
    &=& a_{n-3}(\mathbf{0})+\sum\limits_{k=1}^{n-4}\beta_k\frac{\left( \frac{2}{h(\xi_k)}+o(1)\right) \sum\limits_{j=1}^{n-3-k}t_j }{\left(1+o(1)\right) \ln \left( \frac{1}{ \sum\limits_{j=1}^{n-3-k}t_j }\right)}  
    +4 \sum\limits_{k=1}^{n-4}\frac{\beta_{kk}}{h^2(\xi_k)} \frac{\left(\sum\limits_{j=1}^{n-3-k}t_j
    \right)^2}{\ \ln \left( \frac{1}{ \sum\limits_{j=1}^{n-3-k}t_j }\right)}  
     \\
    &&+ o\left(\sum\limits_{i,j=1}^{n-4}\frac{t_it_j}{\ln\left( \frac{1}{ \sum\limits_{j=1}^{n-4}t_j }\right)}\right)\nonumber
\end{eqnarray} 

On the other hand, we have   $R(\mathbf{t})= \widehat{L}(\mathbf{s})$ in $\Tei_{0,n}$. By Lemma \ref{Lemma 4.2},  $F( R(\mathbf{t}))=F(\widehat{L}(\mathbf{s}))$ is the area of 
    $ L(\mathbf{s})$.
Note that  $a_{n-3}(\mathbf{0})=a_{n-3}$ and 
 $F(S(\mathbf{a}_0,\mathbf{b}))=a_{n-3}$.
It turns out that \eqref{equ:area} is exactly the value of the function $F( R(\mathbf{t}))$.

We know that $F(R(\mathbf{t}))$ is smooth, so is $\eqref{equ:area}$. Thus we may assume that 

\begin{eqnarray*}
   && \sum\limits_{k=1}^{n-4}\beta_k\frac{\left( \frac{2}{h(\xi_k)}+o(1)\right) \sum\limits_{j=1}^{n-3-k}t_j }{\left(1+o(1)\right) \ln \left( \frac{1}{ \sum\limits_{j=1}^{n-3-k}t_j }\right)}  
    + 4 \sum\limits_{k=1}^{n-4}\frac{\beta_{kk}}{h^2(\xi_k)} \frac{\left(\sum\limits_{j=1}^{n-3-k}t_j
    \right)^2}{\ \ln \left( \frac{1}{ \sum\limits_{j=1}^{n-3-k}t_j }\right)} + o\left(\sum\limits_{i,j=1}^{n-4}\frac{t_it_j}{\ln\left( \frac{1}{ \sum\limits_{j=1}^{n-4}t_j }\right)}\right)\\&=&\sum\limits_{k=1}^{n-4}\delta_{k}t_k +\sum\limits_{j,k=1}^{n-4}\delta_{jk}t_jt_k +O\left(\sum\limits_{i,j,k=1}^{n-4}t_it_jt_k\right)\nonumber
\end{eqnarray*}
To derive a contradiction,  if we  let $t_k \to 0$ for $k=n-4, \cdots, k=1$ one by one, then we can Verify that $\delta_{1}=\cdots=\delta_{n-4}=0$ and $\beta_1=\cdots=\beta_{n-4}=0$. 
Then  we have 
\begin{eqnarray*}
4\sum\limits_{k=1}^{n-4}\frac{\beta_{kk}}{h^2(\xi_k)} \frac{\left(\sum\limits_{j=1}^{n-3-k}t_j
    \right)^2}{\ \ln \left( \frac{1}{ \sum\limits_{j=1}^{n-3-k}t_j }\right)} + o\left(\sum\limits_{i,j=1}^{n-4}\frac{t_it_j}{\ln\left( \frac{1}{ \sum\limits_{j=1}^{n-4}t_j }\right)}\right)=\sum\limits_{j=1,k=1}^{n-4}\delta_{jk}t_jt_k +O\left(\sum\limits_{i,j,k=1}^{n-4}t_it_jt_k\right)
\end{eqnarray*}
 Take $x_1=\cdots=x_{n-5}=0$ and then divide both sides of the above equation by $t_{n-4}^2$, we can obtain $\delta_{n-4, n-4}=0$ by taking  $t_{n-4}\to 0$.   Now we keep $x_1=\cdots=x_{n-5}=0$ and divide both sides of the above equation  by $t_{n-4}^3$. Again, we take $t_{n-4}\to 0$, which leads to $\frac{\beta_{11}}{h^2(\xi_1)} =0.$
This is impossible because $\frac{\beta_{11}}{h^2(\xi_1)} >0$.



This completes the proof of Theorem \ref{thm:typeA}.

\subsection{$L$-shaped pillowcases with marked points}\label{sec:L-shape}
We can simplify the proof of Theorem \ref{thm:typeA} by collapsing cylinders. Consider any staircase surface $\phi$. 
 In the vertical direction, there are $n-3$ cylinders. 
By collapsing $n-5$ vertical cylinders from right to left, we obtain an 
$L$-shaped pillowcase with $n-5$ regular marked points located at the right vertical edge of the polygon. We denote such an $L$-shaped pillowcase with marked points by $\varphi$. See Figure \ref{fig:L-shape}. By Proposition \ref{limits}, if the Teichm\"uller disk 
 $\tau^{\phi} : \mathbb{H} \to \Tei_{0,n}$ is a holomorphic retract, so does $\tau^{\varphi}$. 

\begin{figure}[h]
    \centering
    \begin{tikzpicture}[line width=1pt]
        \draw (0,0)--(5,0)--(5,3.9)--(2.7,3.9)--(2.7,5.1)--(0,5.1)--(0,0);
       \draw [dashed] (0,1.3)--(5,1.3);
       \draw [dashed] (0,2.6)--(5,2.6);
        \draw[dashed] (2.7,3.9)--(0,3.9);
        \node at (3.5,4.1){$Q(Q_{n-4})$};
     \node at (5.3,4){$P_{3}$};
        \node at (-0.3,-0.25){$P_{n}$};
        \node at (5.3,-0.3){$P_{n-1}$};
        \node at (5.8,1.3){$P_{n-2}(Q_1)$};
        \node at (5.8, 2.6) {$P_{4}(Q_{n-5})$};
         \node at (2.9,5.3){$P_{2}$};
          \node at (0,5.3){$P_{1}$};
    \end{tikzpicture}
    \caption{The $L$-shaped polygon with marked points.}
    \label{fig:L-shape}
\end{figure}

Assume that $n>5$, and $5 \leq m <n$. By repeating the previous calculation,
we can show that if $\varphi$ is a staircase surface in $\mathcal{Q}(1^{m-4},-1^m, 0^{n-m})$,
    then the Teichm\"uller disk $\tau^\varphi$
    is not a holomorphic retract of $\mathcal{T}_{0,n}$. Then the same argument as in the proof of Proposition \ref{MQ is empty} shows that for any $\varphi\in \mathcal{Q}(1^{m-4},-1^m, 0^{n-m})$ the associated Teichm\"uller disk $\tau^\varphi$
    is not a holomorphic retract of $\mathcal{T}_{0,n}$.
We omit the details.

 In the proof of Theorem \ref{thm:typeA} above, we adopt a direct yet relatively complicated approach, because we believe that it  is useful.

\bibliographystyle{plain}

\bibliography{bibliography}

\end{document}